\documentclass[12pt]{amsart}

\usepackage{latexsym, amssymb, amsmath}

\usepackage{amsfonts, graphicx}
\usepackage{graphicx,color}
\newcommand{\be}{\begin{equation}}
\newcommand{\ee}{\end{equation}}
\newcommand{\beq}{\begin{eqnarray}}
\newcommand{\eeq}{\end{eqnarray}}

\newtheorem{thm}{Theorem}[section]

\newtheorem{lma}{Lemma}[section]

\newtheorem{cor}{Corollary}[section]

\theoremstyle{remark}

\numberwithin{equation}{section}

\def\p{\partial}

\def\C{\mathcal{C}}
\def\R{\mathbb{R}}

\def\p{\partial}
\def\lf{\left}
\def\ri{\right}
\def\e{\epsilon}
\def\ol{\overline}
\def\R{\Bbb R}
\def\wt{\widetilde}
\def\la{\langle}
\def\ra{\rangle}

\def\Ric{\text{\rm Ric}}

\def\Pi{\overline{\displaystyle{\mathbb{II}}}}

\def\a{\alpha}
\def\b{\beta}

\def\heat{\lf(\Delta -\frac{\p}{\p t}\ri)}
\def\K{K\"ahler }

\def\aint{\frac{\ \ }{\ \ }{\hskip -0.4cm}\int}

\def\heat{\lf(\frac{\p}{\p t}-\Delta\ri)}

\def\Ric{\text{Ric}}
\def\lf{\left}
\def\ri{\right}

\def\a{\alpha}
\def\ol{\overline}

\def\e{\epsilon}
\def\p{\partial}

\def\C{\Bbb C}
\def\R{\Bbb R}

\def\bb{{\bar{\beta}}}
\def\abb{\alpha{\bar{\beta}}}

\def\tn{\widetilde \nabla}

\def\aint{\frac{\ \ }{\ \ }{\hskip -0.4cm}\int}

\def\Ric{\operatorname{Ric}}
\def\Rm{\operatorname{Rm}}

\def\K{K\"ahler\ }
\def\be{\begin{equation}}
\def\ee{\end{equation}}

\def\b{\bar}

\def\cd{\nabla}

\def\ppt{\frac{\partial}{\partial t}}

\def\ppza{\frac{\partial}{\partial z^A}}
\def\t{\tilde}
\def\bee
{\begin{equation*}}

\def\eee{\end{equation*}}

\def\n{\nabla}
\def\tn{\wt\nabla}

\def\la{\langle}
\def\ra{\rangle}
\def\bee{\begin{equation*}}
\def\eee{\end{equation*}}
\def\aint{\frac{\ \ }{\ \ }{\hskip -0.4cm}\int}

\def\ol{\overline}
\def\e{\epsilon}
\def\lf{\left}
\def\heat{\lf(\frac{\p}{\p t}-\Delta\ri)}

\def\ri{\right}

\def\a{{\alpha}}
\def\b{{\beta}}

\def\bb{{\bar\b}}

\def\abb{{\a\bb}}

\def\wt{\widetilde}
\def\tn{{\wt\nabla}}
\def\p{\partial}

\def\K{K\"ahler }
\def\KR{K\"ahler-Ricci }
\def\KRF{K\"ahler-Ricci flow }

\def\be{\begin{equation}}
\def\ee{\end{equation}}
\def\ol{\overline}
\def\lf{\left}
\def\ri{\right}
\def\a{{\alpha}}
\def\b{{\beta}}

\def\e{\epsilon}

\def\bb{{\bar\b}}

\def\Ric{\text{\rm Ric}}
\def\Rm{\text{\rm Rm}}

\def\abb{{\a\bb}}

\def\wt{\widetilde}
\def\tn{{\wt\nabla}}
\def\p{\partial}

\def\C{\Bbb C}

\def\wt{\widetilde}
\def\tn{{\wt\nabla}}
\def\p{\partial}

\def\p{\partial}

\def\C{\Bbb C}

\def\KRF{K\"ahler-Ricci flow }

\def\n{\nabla}
\def\tn{\wt\nabla}

\def\cd{\nabla}

\def\ppt{\frac{\partial}{\partial t}}

\def\ppza{\frac{\partial}{\partial z^A}}
\def\t{\tilde}

\def\n{\nabla}
\def\tn{\wt\nabla}

\def\la{\langle}
\def\ra{\rangle}
\def\bee{\begin{equation*}}
\def\eee{\end{equation*}}
\def\aint{\frac{\ \ }{\ \ }{\hskip -0.4cm}\int}

\def\ol{\overline}
\def\e{\epsilon}
\def\lf{\left}
\def\heat{\lf(\frac{\p}{\p t}-\Delta\ri)}

\def\ri{\right}

\def\a{{\alpha}}
\def\b{{\beta}}

\def\bb{{\bar\b}}

\def\abb{{\a\bb}}

\def\wt{\widetilde}
\def\tn{{\wt\nabla}}
\def\p{\partial}

\def\K{K\"ahler }
\def\KR{K\"ahler-Ricci }
\def\KRF{K\"ahler-Ricci flow }

\def\be{\begin{equation}}
\def\ee{\end{equation}}
\def\ol{\overline}
\def\lf{\left}
\def\ri{\right}
\def\a{{\alpha}}
\def\b{{\beta}}

\def\e{\epsilon}

\def\bb{{\bar\b}}

\def\Ric{\text{\rm Ric}}
\def\Rm{\text{\rm Rm}}

\def\abb{{\a\bb}}

\def\wt{\widetilde}
\def\tn{{\wt\nabla}}
\def\p{\partial}

\def\C{\Bbb C}

\def\wt{\widetilde}
\def\tn{{\wt\nabla}}
\def\p{\partial}

\def\p{\partial}

\def\C{\Bbb C}

\def\KRF{K\"ahler-Ricci flow }

\def\n{\nabla}
\def\tn{\wt\nabla}

\def\ppt{\frac{\partial}{\partial t}}

\def\n{\nabla}
\def\tn{\wt\nabla}

\def\la{\langle}
\def\ra{\rangle}
\def\bee{\begin{equation*}}
\def\eee{\end{equation*}}
\def\aint{\frac{\ \ }{\ \ }{\hskip -0.4cm}\int}

\def\ol{\overline}
\def\e{\epsilon}
\def\lf{\left}
\def\heat{\lf(\frac{\p}{\p t}-\Delta\ri)}

\def\ri{\right}

\def\a{{\alpha}}
\def\b{{\beta}}

\def\bb{{\bar\b}}

\def\abb{{\a\bb}}

\def\wt{\widetilde}
\def\tn{{\wt\nabla}}
\def\p{\partial}

\def\K{K\"ahler }
\def\KR{K\"ahler-Ricci }
\def\KRF{K\"ahler-Ricci flow }

\def\be{\begin{equation}}
\def\ee{\end{equation}}
\def\ol{\overline}
\def\lf{\left}
\def\ri{\right}
\def\a{{\alpha}}
\def\b{{\beta}}

\def\e{\epsilon}

\def\bb{{\bar\b}}

\def\Ric{\text{\rm Ric}}
\def\Rm{\text{\rm Rm}}

\def\abb{{\a\bb}}

\def\wt{\widetilde}
\def\tn{{\wt\nabla}}
\def\p{\partial}

\def\C{\Bbb C}

\def\wt{\widetilde}
\def\tn{{\wt\nabla}}
\def\p{\partial}

\def\p{\partial}

\def\C{\Bbb C}

\def\KRF{K\"ahler-Ricci flow }

\def\n{\nabla}
\def\tn{\wt\nabla}

\def\cd{\nabla}

\def\ppt{\frac{\partial}{\partial t}}

\def\ppza{\frac{\partial}{\partial z^A}}
\def\t{\tilde}

\def\n{\nabla}
\def\tn{\wt\nabla}

\def\la{\langle}
\def\ra{\rangle}
\def\bee{\begin{equation*}}
\def\eee{\end{equation*}}
\def\aint{\frac{\ \ }{\ \ }{\hskip -0.4cm}\int}

\def\ol{\overline}
\def\e{\epsilon}
\def\lf{\left}
\def\heat{\lf(\frac{\p}{\p t}-\Delta\ri)}

\def\ri{\right}

\def\a{{\alpha}}
\def\b{{\beta}}

\def\bb{{\bar\b}}

\def\abb{{\a\bb}}

\def\wt{\widetilde}
\def\tn{{\wt\nabla}}
\def\p{\partial}

\def\K{K\"ahler }
\def\KR{K\"ahler-Ricci }
\def\KRF{K\"ahler-Ricci flow }

\def\be{\begin{equation}}
\def\ee{\end{equation}}
\def\ol{\overline}
\def\lf{\left}
\def\ri{\right}
\def\a{{\alpha}}
\def\b{{\beta}}

\def\e{\epsilon}

\def\bb{{\bar\b}}

\def\Ric{\text{\rm Ric}}
\def\Rm{\text{\rm Rm}}

\def\abb{{\a\bb}}

\def\wt{\widetilde}
\def\tn{{\wt\nabla}}
\def\p{\partial}

\def\C{\Bbb C}

\def\wt{\widetilde}
\def\tn{{\wt\nabla}}
\def\p{\partial}

\def\p{\partial}

\def\C{\Bbb C}

\def\KRF{K\"ahler-Ricci flow }

\def\n{\nabla}
\def\tn{\wt\nabla}

\def\ppt{\frac{\partial}{\partial t}}
\def\pptk{\frac{\partial^k}{\partial t^k}}

\begin{document}
\title{  K\"ahler-Ricci flow with  unbounded curvature}

\author{Shaochuang Huang}

\address{Department of Mathematics,
The Institute of Mathematical Sciences and Department of
 Mathematics, The Chinese University of Hong Kong,
Shatin, Hong Kong, China.} \email{schuang@math.cuhk.edu.hk}

\author{Luen-Fai Tam$^1$}
\address{The Institute of Mathematical Sciences and Department of
 Mathematics, The Chinese University of Hong Kong,
Shatin, Hong Kong, China.} \email{lftam@math.cuhk.edu.hk}
\thanks{$^1$Research partially supported by Hong Kong  RGC General Research Fund
\#CUHK 1430514}
 \renewcommand{\subjclassname}{
  \textup{2010} Mathematics Subject Classification}
\begin{abstract}
Let $g(t)$ be a   complete solution to the Ricci flow on a   noncompact manifold such that $g(0)$ is \K. We   prove that if $|\Rm(g(t))|_{g(t)}\le a/t$ for some $a>0$, then $g(t)$ is \K for $t>0$.  We prove that there is a constant $ a(n)>0$ depending only on $n$ such that the following is true: Suppose $g(t)$ is a complete solution to the \KR flow on a noncompact $n$-dimensional  complex manifold such that  $g(0)$ has nonnegative holomorphic bisectional curvature and such that $|\Rm(g(t))|_{g(t)}\le a(n)/t$, then $g(t)$ has nonnegative holomorphic bisectional curvature for $t>0$. These generalize the results in \cite{Shi1997}. As corollaries, we prove that (i) any complete noncompact \K manifold with nonnegative complex sectional curvature with maximum volume growth is biholomorphic to $\C^n$; and (ii) there is $\e(n)>0$ depending only on $n$ such that if $(M^n,g_0)$ is a complete noncompact \K manifold of complex dimension $n$
with nonnegative holomorphic bisectional curvature and maximum volume growth and if $(1+\e(n))^{-1}h\le g_0\le (1+\e(n))h$ for some Riemannian metric $h$ with bounded curvature, then $M$ is biholomorphic to $\C^n$. \vskip .1cm


\end{abstract}

\date{June, 2015}
\maketitle
\markboth{Shaochuang Huang and Luen-Fai Tam} {K\"ahler Ricci flow with unbounded curvature}
\section{Introduction}

In \cite{Simon2002},   Simon proved that there is a constant $\e(n)>0$ depending only on $n$ such that if $(M^n,g_0)$ is a complete $n$ dimensional Riemannian manifold and if there is another metric $h$ with curvature bounded by $k_0$ with $$(1+\e(n))^{-1}h\le g_0\le (1+\e(n))h,$$
then the so-called $h$-flow has a short time solution $g(t)$ such that
\be\label{e-Rbound}
|\Rm(g(t))|_{g(t)}\le C/t.
 \ee
 Here $h$-flow is basically  the usual Ricci-DeTurck flow. If $h=g_0$, the $h$-flow is exactly the Ricci-DeTurck flow. For the precise definition of $h$-flow, see Section \ref{s-hflow}.  It is not hard to construct Ricci flow using the solution of $h$-flow if $g_0$ is smooth. On the other hand, in \cite{Cabezas-RivasWilking2011},
 Cabezas-Rivas and Wilking  proved that if $(M,g_0)$ is a complete noncompact Riemannian manifold  with nonnegative complex sectional curvature, and if the volume of geodesic ball $B(x,1)$ of radius 1 with center at $x$ is uniformly bounded below away from 0, then the Ricci flow have a solution for short time with nonnegative complex sectional curvature so that \eqref{e-Rbound} holds.
 Recall that a Riemannian manifold is said to have nonnegative complex sectional curvature if $R(X,Y,\bar Y,\bar X)\ge0$ for any vectors in the complexified tangent bundle.

 It is natural to ask the following:

  {\it Question:  Suppose  $g_0$ is K\"ahler. Are the above solutions $g(t)$ of Ricci flow also \K for  $t>0$?}

 This question has been studied before. It was proved by
Yang and Zheng \cite{YangZheng2013}  for a   $U(n)$ invariant initial \K metric on $\C^n$,  the solution constructed by Cabezas-Rivas and Wilking is \K for $t>0$, under some  additional technical  conditions.

It is well-known that if $M$ is compact or if the curvature of $g_0$ is bounded,  the answer to the above question is yes by \cite{Hamilton1982} and \cite{Shi1997}.
In this paper, we want to prove the following:

\begin{thm}\label{t-intro-1} If $(M^n,g_0)$ is a complete noncompact \K manifold with complex dimension $n$ and if $g(t)$ is a smooth complete solution to the Ricci flow on $M\times[0,T]$, $T>0$, with $g(0)=g_0$ such that
\bee\label{e-curvaturebound}
 |\Rm(g(t))|_{g(t)}\le \frac at
\eee for some $a>0$,
 then $g(t)$ is \K for all $0\le t\le T$.
 \end{thm}
This gives an affirmative answer to the above question. The result is related to previous  works on the   on the existence of \KR flows without curvature bound, see \cite{ChauLiTam1,ChauLiTam2,GT,YangZheng2013} for example.

We may apply the theorem to the uniformization conjecture by Yau \cite{Y} which states that  a complete noncompact
K\"ahler manifold with positive holomorphic bisectional curvature is
biholomorphic to $\C^n$. A previous result by Chau and the second author \cite{ChauTam2008}   says that the conjecture is true if the \K manifold has maximum volume growth and has {\it bounded  curvature}, see also   \cite{Mok1984,CTZ}.
Combining this with the theorem, we have:

\begin{cor}\label{c-intro-1}
Let $(M^n,g_0)$ be a complete noncompact \K manifold with complex dimension $n$ and with nonnegative complex sectional curvature. Suppose $M^n$ has maximum volume growth. Then $M^n$ is biholomorphic to $\C^n$.
\end{cor}
 For \K surface, sectional curvature being nonnegative is equivalent to complex sectional curvature being nonnegative \cite{Zheng1995}. Hence in particular, any complete  \K surface with nonnegative sectional curvature with maximum volume growth is biholomorphic to $\C^2$.
  We should mention  that recently Liu \cite{Liu2015} proves that   a complete noncompact \K manifold with nonnegative holomorphic bisectional curvature and with maximum volume growth is biholomorphic to an affine algebraic variety, generalizing the result of Mok \cite{Mok1984}. Moreover, if the volume of geodesic balls are close  to the Euclidean balls with same radii, then the manifold is biholomorphic to $\C^n$.

By Theorem \ref{t-intro-1}, we know that from the solution constructed by   Simon \cite{Simon2002} one can construct a solution to the \KR flow if $g_0$ is \K. In view of the conjecture of Yau, we would like to know that if the nonnegativity of holomorphic bisectional curvature will be preserved by the solution $g(t)$ of the \KR flow. The second result in this paper is the following:

\begin{thm}\label{t-intro-2} There is $0< a(n)<1$ depending only on $n$ such that if $g(t)$ is a complete solution of  \KRF on $M\times[0,T]$ with $|\Rm(g(t))|_{g(t)}\leq \frac{a}{t}$, where $M$ is an $n$-dimensional non-compact complex manifold. If $g(0)$ has nonnegative holomorphic bisectional curvature, then so does $g(t)$ for all $t\in[0,T]$.
\end{thm}
We should mention that in \cite{YangZheng2013}, Yang and Zheng proved that the nonnegativity of bisectional curvature is preserved under the \KR flow for $U(n)$ invariant solution on $\C^n$ without any condition on the bound of the curvature.

By refining the estimates in \cite{Simon2002}, one can prove that if $\e(n)>0$ is small in the result of Simon, then curvature   of the solution of the $h$-flow will be bounded by $a/t$ with $a$ small. Hence as a corollary to the theorem, using \cite{ChauTam2008} again, we have:
\begin{cor}\label{c-intro-2} There is $\e(n)>0$, depending only on $n$. Suppose $(M^n,g_0)$ is a complete noncompact \K manifold with complex dimension $n$ with nonnegative holomorphic bisectional curvature with maximum volume growth. Suppose there is a Riemannian metric $h$ on $M$ with bounded curvature such that   $(1+\e(n))^{-1}h\le g_0\le (1+\e(n))h$. Then $M$ is biholomorphic to $\C^n$.
\end{cor}

By a result of Xu \cite{Xu2013}, we also have the following corollary which says that the condition that the curvature is bounded in the uniformization result in \cite{ChauTam2008} can be relaxed to the condition that the curvature is bounded in some integral sense. Namely, we have:

\begin{cor}\label{c-intro-3} Let   $(M^n,g_0)$ is a complete noncompact \K manifold with complex dimension $n\ge2$ with nonnegative holomorphic bisectional curvature with maximum volume growth. Suppose there is $r_0>0$ and there is $C>0$ such that
$$
\lf(\aint_{B_x(r_0)}|\Rm|^p\ri)^\frac1p\le C
$$
for some $p>n$ for all $x\in M$. Then $M$ is biholomorphic to $\C^n$.
\end{cor}

The paper is organized as follows: in Section \ref{s-max} we prove a maximum principle and apply it in Section \ref{s-kahler} to prove Theorem \ref{t-intro-1}. In Section \ref{s-bisectional} we prove Theorem \ref{t-intro-2}. In Section \ref{s-hflow}, we will construct solution to the \KR flow with nonnegative holomorphic bisectional curvature through the $h$-flow.

{\it Acknowledgement}: The second author would like to thank Albert Chau for some usefully discussions and for bringing our attention to the results in \cite{Xu2013}.

\section{A maximum principle}\label{s-max}

 In this section, we will prove a maximum principle, which will be used in the proof of Theorem \ref{t-intro-1}.

 Let $(M^n,g_0)$ be a complete noncompact Riemannian manifold. Let $g(t)$ be a smooth complete solution to the Ricci flow on $M\times[0,T]$, $T>0$ with $g(0)=g_0$, i.e.

  \be\label{e-Riccifloweq}
  \left\{
    \begin{array}{ll}
      \ppt g\,\,=-2\Ric,&\hbox{on $M\times[0,T]$;} \\
      g(0)=g_0.&\hbox{ }
    \end{array}
  \right.
  \ee
Let $\Gamma$ and $\bar \Gamma$ be the Christoffel symbols of $g(t)$ and $\bar g=g(T)$ respectively. Let $A=\Gamma-\bar\Gamma$. Then $A$ is a $(1,2)$ tensor.
In the following, lower case $c, c_1, c_2,\dots$ will denote positive constants depending only on $n$.

\begin{lma} \label{l-connection} With the above notation and assumptions, suppose the curvature satisfies $|\Rm(g(t)|_{g(t)}\le at^{-1}$ for some positive constant $a$.  Then there is a constant $c=c(n)>0$, such that
 \begin{enumerate}
 \item [(i)]
\bee\label{e-metric}
\ \lf(\frac{T}{t}\ri)^{-ca}\bar g\le g(t)\le  \lf(\frac{T}{t}\ri)^{ca}\bar g;
\eee

   \item [(ii)] $|\nabla \Rm|\le   Ct^{-\frac32}$ for some constant $C=C(n,T,a)>0$ depending only on $n, T, a$;
   \item [(iii)]  $$|A|_{\bar g}\le Ct^{-\frac12-ca},$$
       for some constant $C=C(n,T,a)>0$ depending only on $n, T$ and $a$.
 \end{enumerate}
\end{lma}
\begin{proof} (i) follows from the Ricci flow equation.

(ii) is a result in \cite{Shi1989}, see also \cite[Theorem 7.1]{ChowKnopf}.

To prove (iii), in local coordinates:
\bee
\frac{\p }{\p t}A_{ij}^k=-g^{kl}\lf(\nabla_i R_{jl}+\nabla_jR_{il}-\nabla_l R_{ij}\ri).
\eee

At a point where $\bar g_{ij}=\delta_{ij}$ such that $g_{ij}=\lambda_i\delta_{ij}$
\bee
\begin{split}
\lf|\frac{\p }{\p t}|A|_{\bar g}^2\ri|\le&C_1(n,T) t^{-c_1a}  |\nabla\Ric|_{\bar g}|A|_{\bar g}\\
\le& C_2(n,T,a)t^{-c_2a-\frac32}|A|_{\bar g}\\
\end{split}
\eee
for some constants $C_1, C_2$ depending only on $n, T, a$ and $c_1, c_2$ depending only on $n$. From this the result follows.
\end{proof}

Under the assumption of the lemma, since $g(T)$ is complete and the curvature of $\bar g=g(T)$ is bounded by $a/T$,  we can find a smooth function  $\rho$ on $M$ such that
\be\label{e-exhaustionfn}
d_{\bar g}(x,x_0)+1 \le \rho(x)\le C'(d(x,x_0)+1);\ |\bar\nabla \rho|_{\bar g}+|\bar\nabla^2 \rho|_{\bar g}\le C'
\ee
for some $C'>1$, where  $\bar \nabla $ is covariant derivative with respect to $\bar g$ and $C'>0$ is a constant depending on $n$ and $a/T$,  see \cite{Shi1997,Tam2010}.

\begin{lma}\label{l-cutoff} With the same assumptions and notation as in the previous lemma, $\rho(x)$ satisfies
\bee
|\nabla \rho|\le   C_1  t^{-ca}
\eee
and
\bee
|\Delta \rho|\le C_2 t^{-\frac12-ca}
\eee
where $C_1$, $C_2$ depending only on $n, T, a$ and $c>0$ depending only on $n$. Here $\nabla$ and $\Delta$ are the covariant derivative and Laplacian of $g(t)$ respectively.
\end{lma}
\begin{proof} The first inequality follows from Lemma \ref{l-connection}(i). To estimate $\Delta \rho$, at a point where $\bar g_{ij}=\delta_{ij}$ and $g_{ij}$ is diagonalized, we have
\bee
\begin{split}
\lf|\Delta \rho-\bar\Delta\rho\ri|=& \lf|g^{ij}\nabla_i\nabla_j\rho-g_T^{ij}\ol\nabla_i\ol\nabla_j\rho\ri|\\
\le&\lf|g^{ij}\lf(\nabla_i\nabla_j-\ol\nabla_i\ol\nabla_j\ri)\rho\ri|+
\lf|\lf(g^{ij}-g_T^{ij}\ri)\ol\nabla_i\ol\nabla_j\rho\ri|\\
\le &|g^{ij}A_{ij}^k\rho_k|+C_3t^{-c_1a}\\
\le& C_4  t^{-\frac12-c_2a}
\end{split}
\eee
for some constants $C_3$, $C_4$ depending only on $n, T, a$, and $c_1, c_2$ depending only on $n$.
By the estimates of $\bar \Delta\rho$, the second result follows.
\end{proof}

\begin{lma}\label{l-maximum}
Let $(M^n,g)$ be a complete noncompact Riemannian manifold with dimension $n$ and let  $g(t)$ be a smooth complete solution of the Ricci flow on $M\times[0,T]$, $T>0$ such that the curvature satisfies $|\Rm|\le at^{-1}$ for some $a>0$.

Let $f\ge0$ be a smooth function on $M\times[0,T]$  such that
\begin{enumerate}
         \item[(i)] $$
\heat f\le \frac at f;
$$

         \item [(ii)] $\frac{\p^k f}{\p t^k}|_{t=0}=0$ for all $k\ge 0$;
         \item[(iii)] $\sup_{x\in M}f(x,t)\le Ct^{-l}$, for some positive integer $l$ for some constant $C$.
       \end{enumerate}
       Then $f\equiv0$ on $M\times[0,T]$.
\end{lma}
\begin{proof} We may assume that $T\le 1$. In fact, if we can prove that $f\equiv0$ on $M\times[0,T_1]$ where $T_1=\min\{1,T\}$, then it is easy to see that $f\equiv0$ on $M\times[0,T]$ because $f$ and the curvature of $g(t)$ are uniformly bounded on $M\times[T_1,T]$.

Let $p\in M$ be a fixed point, and let $d(x,t)$ be the distance between $p, x$ with respect to $g(t)$. By \cite{Perelman-1} (see also \cite[Chapter 18]{Chow07}),   for all $r_0$, if $d(x,t)> r_0$, then

\be
\frac{\p_-}{\p t}d(x,t)-\Delta_t d(x,t)\ge - C_0\lf(t^{-1}r_0+\frac1 {r_0}\ri)
\ee
in the barrier sense, for some $C_0=C_0(n, a)$ depending only on $n$ and $a$. Here
\be
\frac{\p_-}{\p t}d(x,t)=\liminf_{h\to 0^+}\frac{d(x,t)-d(x,t-h)}h.
\ee
The above inequality means  that for any $\e>0$, there is a smooth function $\sigma(y)$ near $x$ such that $\sigma(x)=d(x,t)$, $\sigma(y)\ge d(y,t)$ near $x$, such that $\sigma$ is $C^2$ and
\be
\frac{\p_-}{\p t}d(x,t)-\Delta_t \sigma(x)\ge -C_0\lf(t^{-1}r_0+\frac1 {r_0}\ri)-\e.
\ee
 In the following, we always take $\e=T^{-\frac12}$.

Let $f$ be as in the lemma. First we want to prove that for any integer $k>0$ there is a constant $B_k$ such that
\be
\sup_{x\in M}f(x,t)\le B_k t^k.
\ee

We may assume that $k>a$. Let $F=t^{-k}f$, then
\be
\heat F\le -\frac{k-a}t F\le 0.
\ee
Let $1\ge \phi\ge 0$ be a smooth function on $[0,\infty)$ such that
\bee
\phi(s)=\left\{
  \begin{array}{ll}
    1, & \hbox{if $0\le s\le 1$;} \\
    0, & \hbox{if $s\ge 2$,}
  \end{array}
\right.
\eee
and such that $-C_1\le \phi'\le 0$, $|\phi''|\le C_1$ for some $C_1>0$.
Let $\Phi=\phi^m$, where $m> 2$ will be chosen later. Then $\Phi=1$ on $[0,1]$ and $\Phi=0$ on $[2,\infty)$, $1\ge\Phi\ge0$, $-C(m)\Phi^{q}\le \Phi'\le 0$, $|\Phi''|\le C(m)\Phi^{q}$.
where $C(m)>0$ depends on $m$ and $C_1$, and $q=1-\frac 2m$.

For any $r>>1$, let $\Psi(x,t)=\Phi(\frac{d(x,t)}r)$. Let
$$\theta(t)=\exp(-\a t^{1-\b}),
$$
where $\a>0$, $0<\b<1$ will be chosen later.

We claim that one can choose $m$, $\a$ and $\b$ such that for all $r>>1$
$$
H(x,t)=\theta(t)\Psi(x,t)F\le C_2
$$
 on $M\times[0,T]$, where $C_2$ is independent of $r$. If the claim is true, then we have
 $F$ is bounded. Hence $f(x,t)\le B_k t^k$.

First note that $\Psi(x,t)$ has compact support in $M\times[0,T]$.
By assumption (ii) and the fact $f$ is smooth, we conclude that $H(x,t)$ is continuous on $M\times[0,T]$. Moreover, by (ii) again, $H(x,0)=0$. Suppose $H(x,t)$ attains a positive maximum at $(x_0,t_0)$ for some $x_0\in M$, $t_0>0$. Suppose $d(x_0,t_0)<r$, then there is a neighborhood $U$ of $x$ and $\delta>0$ such that $d(x,t)<r$ for $x\in U$ and $|t-t_0|<\delta$. For such $(x,t)$, $H(x,t)=\theta(t)F(x,t)$. Since $H(x_0,t_0)$ is a local maximum, we have
\bee
\begin{split}
0\le &\heat H\\
=& \theta(t)\lf(\theta'F+\heat F\ri)\\
<&0
\end{split}
\eee
which is a contradiction.

Hence we must have $d(x_0,t_0)\ge r$. If $r>>1$, then $r\ge T^\frac12$, and at $(x_0,t_0)$
\bee
\frac{\p_-}{\p t}d(x,t)-\Delta_t d(x,t)\ge - (2C_0+1) t^{-\frac12 }
\eee
in the barrier sense, by taking $r_0=t^\frac12$.
 Let $\sigma(x)$ be a barrier function near $x_0$. Let $\wt\Psi(x)=\Phi(\frac{\sigma(x)}r)$, and let
$$
\tilde H(x,t)=\theta(t)\tilde \Psi(x)F(x,t)
$$
which is defined near $x_0$ for all $t$. Moreover,
$$
\tilde H(x_0,t_0)=H(x_0,t_0)
$$
and
$$
\tilde H(x,t_0)\le H(x,t_0)
$$
near $x_0$ because $\sigma(x)\ge d(x,t_0)$ near $x_0$ and $\Phi'\le0$. Hence $\wt H(x,t_0)$ has a local maximum at $(x_0,t_0)$ as a function of $x$. So we have
\be
\nabla \tilde H(x_0,t_0)=0
\ee
 and
 \be
 \Delta \tilde H(x_0,t_0)\le0.
 \ee
  At $(x_0,t_0)$
 \be\label{e-max-1}
 \begin{split}
 0\ge&  \Delta \lf(\theta(t)\wt\Psi(x)F (x,t)\ri)\\
 =&\theta\wt\Psi \Delta F +\theta F  \Delta \wt\Psi+2\theta\la\nabla F ,\nabla \wt\Psi\ra \\
 =& \theta\wt\Psi  \Delta F+ \theta F   \lf( \frac1r\Phi'\Delta \sigma+\frac1{r^2}\Phi''|\nabla \sigma|^2\ri)-2\theta \frac{|\nabla \wt\Psi|^2}{\wt\Psi}F \\
 \ge &\theta \Phi  \Delta F+ \theta F   \lf( \frac1r\Phi'\Delta \sigma+\frac1{r^2}\Phi''|\nabla \sigma|^2\ri)- \frac2{r^2}\theta \frac{\Phi'^2}{\Phi}F\\
 \end{split}
 \ee
 where we have used the fact that $\sigma(x)\ge d(x,t_0)$ near $x_0$ and $\sigma(x_0)=d(x_0,t_0)$ so that $|\nabla \sigma(x_0)|\le 1$. $\Phi$ and the derivatives $\Phi'$ and $\Phi''$ are evaluated at $\frac{d(x_0,t_0)}r$.

 On the other hand,
 \bee
 \begin{split}
 0\le& \liminf_{h\to 0^+}\frac{H(x_0,t_0)-H(x_0,t_0-h)}h\\
 =&\theta'\Psi F+\theta\Psi \frac{\p}{\p t} F+\theta F\liminf_{h\to 0^+}\frac{-\Psi(x_0,t_0-h)+\Psi(x_0,t_0)}h.
 \end{split}
 \eee
 Now
 \bee
 \begin{split}
 -\Psi(x_0,t_0-h)+\Psi(x_0,t_0)=&  -\Phi(\frac{d(x_0,t_0-h)}r)+\Phi(\frac{d(x_0,t_0)}r)\\
 =& \frac1r\Phi'(\xi)(d(x_0,t_0)-d(x_0,t_0-h)),
 \end{split}
 \eee
 for some $\xi$ between $\frac1rd(x_0,t_0-h)$ and $\frac1r d(x_0,t_0)$
which implies
\bee
\begin{split}
\liminf_{h\to 0^+}\frac{-\Psi(x_0,t_0-h)+\Psi(x_0,t_0)}h\le&\limsup_{h\to 0^+}\frac{-\Psi(x_0,t_0-h)+\Psi(x_0,t_0)}h\\
=&  \frac1r\Phi' \frac{\p_-}{\p t}d(x_0,t)|_{t=t_0}
\end{split}
\eee
because $\Phi'\le 0$, where $\Phi'$ is     evaluated at $\frac1r d(x_0,t_0)$.   In the following, $C_i$ will denote positive constants independent of $\a, \b$. Combining the above inequality  with \eqref{e-max-1}, we have at $(x_0,t_0)$:
\bee
 \begin{split}
 0\le&\theta'\Phi F+\theta\Phi \frac{\p}{\p t} F+\theta F\frac1r\Phi' \frac{\p_-}{\p t}d(x_0,t_0)\\
 &-\theta\Psi  \Delta F- \theta F   \lf( \frac1r\Phi'\Delta \sigma+\frac1{r^2}\Phi''|\nabla \sigma|^2\ri)+ \frac2{r^2}\theta \frac{\Phi'^2}{\Phi}F\\
 \le &\theta'\Phi F+C_2\lf(t_0^{-\frac12}\Phi^q+ \Phi^{2q-1}\ri)\theta F\\
 \le &-\a(1-\b)t^{-\b}_0\theta \Phi F+C_3\theta \lf[t^{-\frac12}_0 t^{-(1-q)(k+l)}_0(\Phi F)^q+t^{-\frac12}_0 t^{-2(1-q)(k+l)}_0(\Phi F)^{2q-1}\ri]\\
 \le &\theta\bigg[-\a(1-\b)t^{-\b}_0  \Phi F+C_4t^{-\frac12-2(1-q)(k+l)}_0 \lf( (\Phi F)^q+ (\Phi F)^{2q-1}\ri)\bigg]
   \end{split}
 \eee
where $\Phi, \Phi', \Phi''$ are evaluated at $d(x_0,t_0)/r$.
Now first choose $m$ large enough depending only on $k, l$ so that $ \frac12+2(1-q)(k+l)=\b<1$. Then choose $\a$ such that $\a(1-\b)>2C_4$. Then one can see that we must have $\Phi F\le 1$. Hence $H=\theta\Phi F\le C$ at the maximum point of $H(x,t)$, where $C$ is a constant independent of $r$. This completes the proof of the claim.

Next, let $F=t^{-a}f$. Then

\bee
\heat F\le 0.
\eee

Let $\rho$ be the function in   Lemma \ref{l-cutoff}, we have
 $$
 |\Delta\rho|\le C_5t^{-b}
 $$
 for some $b>1$. Let $\eta(x,t)=\rho(x)\exp(\frac{2C_5}{1-b}t^{1-b}) $. Note that $\eta(x,0)=0.$

\bee
\begin{split}
\heat \eta=&\exp(\frac{2C_5}{1-b}t^{1-b})\lf(2C_5t^{-b}\rho- \Delta \rho\ri)\\
\ge&C_5t^{-b}\exp(\frac{2C_5}{1-b}t^{1-b}) \\
>&0.
\end{split}
\eee
where we have used the fact that $\rho\ge1$. Since $F\le C_6t^2$ in $M\times[0,T]$. In particular it is bounded. Then for any $\e>0$
\bee
\heat (F-\e\eta-\e t)< 0.
\eee
There is $t_1>0$ depending only on $\e, C_6$ such that $F-\e t<0$ for $t\le t_1$. For $t\ge t_1$, $F-\e\eta<0$ outside some compact set. Hence if $F-\e\eta-\e t>0$ somewhere, then there exist $x_0\in M$, $ t_0>0$ such that $F-\e \eta-\e t$ attains maximum. But this is impossible. So $F-\e\eta-\e t\le0$. Let $\e\to0$, we have $F=0$.
\end{proof}

\section{preservation of the K\"ahler condition}\label{s-kahler}

 In this section, we want to prove Theorem \ref{t-intro-1}  and give some applications. Recall Theorem \ref{t-intro-1} as follows:

 \begin{thm}\label{t-Kahler} If $(M^n,g_0)$ is a complete noncompact \K manifold with complex dimension $n$ and if $g(t)$ is a smooth complete solution to the Ricci flow \eqref{e-Riccifloweq} on $M\times[0,T]$, $T>0$, with $g(0)=g_0$ such that
 $$
 |\Rm(g(t))|_{g(t)}\le \frac at
 $$ for some $a>0$,
 then $g(t)$ is \K for all $0\le t\le T$.
 \end{thm}

 We will use the setup as in \cite[Section 5]{Shi1997}. Let $T_\C M=T_\R M\otimes_{\mathbb{R}} \C$ be the complexification of $T_\R M$, where $T_\R M$ is the real tangent bundle. Similarly, let $T^*_\C M=T_\R^*(M)\otimes_{\mathbb{R}} \C$, where $T^*_\R M$ is the real cotangent bundle.
 Let $z=\{z^1,z^2,\cdots,z^n\}$ be a local holomorphic coordinate on $M$, and \bee
\left\{\begin{array}{ll}
         z^k=x^k+\sqrt{-1}x^{k+n} &  \\
         x^k\in \mathbb{R}, x^{k+n}\in\mathbb{R}, & k=1,2,\cdots,n.
       \end{array} \right. \eee
In the following:
 \begin{itemize}
   \item  $i,j,k,l,\cdots$  denote the indices corresponding to real vectors or real covectors;
   \item $\alpha,\beta,\gamma,\delta,\cdots$   denote the indices corresponding to holomorphic vectors or holomorphic covectors,
   \item $A,B,C,D,\cdots$   denote both $\alpha,\beta,\gamma,\delta,\cdots$ and $\bar\alpha,\bar\beta,\bar\gamma,\bar\delta,\cdots$.
 \end{itemize}
 Extend $g_{ij}(t)$, $R_{ijkl}(t)$ etc. $\C$-linearly to the complexified bundles.
 We have: \bee
\overline{g_{AB}}=g_{\bar A\bar B},\hspace{0.5cm} \overline{R_{ABCD}}=R_{\bar{A}\bar{B}\bar{C}\bar{D}}. \eee
In our convention, $R_{1221}=R(e_1,e_2,e_2,e_1)$ is the sectional curvature of the two-plane spanned by orthonormal pair $e_1, e_2$. $R_{ABCD}$ has the same symmetry as $R_{ijkl}$ and it satisfies the Binachi identities.

Let
$g^{AB}:=(g^{-1})^{AB}$, it means $g^{AB}g_{BC}=\delta^A_C$, and let
 \bee R_{AB}=g^{CD}R_{ACDB} \eee on $M\times[0,T]$.
 Then we have \be\label{e-Ricci-U}
\ppt g_{AB}=-2R_{AB} \ee   and
 \be\label{e-1}\begin{split}
\ppt R_{ABCD}=&\triangle R_{ABCD}-2g^{EF}g^{GH}R_{EABG}R_{FHCD}-2g^{EF}g^{GH}R_{EAGD}R_{FBHC}\\
&+2g^{EF}g^{GH}R_{EAGC}R_{FBHD}-g^{EF}(R_{EBCD}R_{FA}+R_{AECD}R_{FB}\\
&+R_{ABED}R_{FC}+R_{ABCE}R_{FD})  \end{split} \ee
on $M\times[0,T]$.

We begin with the following lemma:

\begin{lma}\label{l-timederivative-1} Let $(M,g_0)$ be a \K manifold, and $g(t)$ be a smooth solution to the Ricci flow with $g(0)=g_0$. In the above set up, we have
\bee
\frac{\p^k}{\p t^k}R_{AB\gamma\delta}|_{t=0}=0
\eee
at each point of $M$ and for all $k\ge 0$ and for all $A, B, \gamma,\delta$.
\end{lma}
\begin{proof} Let $p\in M$ with holomorphic local coordinate $z$.
In the following, all computations are at $(z,0)$ unless we have   emphasis otherwise. We will prove the lemma by induction. Consider the following statement:
 \bee H(k)\left\{\begin{array}{l}
 H_1(k): \pptk R_{AB\gamma\delta}=0 \\
 H_2(k): \pptk g_{AB}=0 \hspace{.1cm} \text{if $A,B$ are of the same type} \\
 H_3(k): \pptk R_{AB}=0 \hspace{.1cm} \text{if $A,B$ are of the same type}\\
 H_4(k): \pptk\Gamma_{AB}^{C}=0 \hspace{.1cm} \text{unless  $A,B,C$ are of the same type} \\
 H_5(k): \pptk R_{AB\gamma\delta;E}=0 \\
 H_6(k): \pptk R_{AB\gamma\delta;EF}=0 \\
 H_7(k): \pptk \triangle R_{AB\gamma\delta}=0
 \end{array} \right. \eee
Here we denote covariant derivative with respect to $g(t)$ by $``;"$ and the partial derivative by $``,"$. If $H_i(k)$ are true for all   $i=1,\cdots,7$, we will say that $H(k)$ holds. As usual:
$$\Gamma_{AB}^C=\frac12 g^{CD}\lf(g_{AD,B}+g_{DB,A}-g_{AB,D}\ri).
$$
We now consider the case that $k=0$. Since the initial metric is K\"hler, it is easy to see that $H(0)$ holds. Now we assume $H(i)$ holds for all $i=0, 1,2,\cdots,k$. We want to show $H(k+1)$ holds. We first see that \bee\begin{split}
&\frac{\partial^{k+1}}{\partial t^{k+1}}R_{AB\gamma\delta}\\
=&\pptk(\triangle R_{AB\gamma\delta})-\sum\limits_{\tiny{\begin{array}{c}
m+n+p+q=k \\
0\le m,n,p,q\le k
\end{array}}
}2(g^{EF})_m(g^{GH})_n(R_{EABG})_p(R_{FH\gamma\delta})_q\\
&-\sum\limits_{\tiny{\begin{array}{c}
m+n+p+q=k \\
0\le m,n,p,q\le k
\end{array}}
}2(g^{EF})_m(g^{GH})_n(R_{EAG\delta})_p(R_{FBH\gamma})_q\\
&+\sum\limits_{\tiny{\begin{array}{c}
m+n+p+q=k \\
0\le m,n,p,q\le k
\end{array}}
}2(g^{EF})_m(g^{GH})_n(R_{EAG\gamma})_p(R_{FBH\delta})_q\\
&-\sum\limits_{\tiny{\begin{array}{c}
m+n+p=k \\
0\le m,n,p\le k
\end{array}}
}(g^{EF})_m(R_{AF})_n(R_{EB\gamma\delta})_p
 -\sum\limits_{\tiny{\begin{array}{c}
m+n+p=k \\
0\le m,n,p\le k
\end{array}}
}(g^{EF})_m(R_{BF})_n(R_{AE\gamma\delta})_p\\
& -\sum\limits_{\tiny{\begin{array}{c}
m+n+p=k \\
0\le m,n,p\le k
\end{array}}
}(g^{EF})_m(R_{\gamma F})_n(R_{ABE\delta})_p
 -\sum\limits_{\tiny{\begin{array}{c}
m+n+p=k \\
0\le m,n,p=\le k
\end{array}}
}(g^{EF})_m(R_{\delta F})_n(R_{AB\gamma E})_p.\end{split}
\eee
Here $(\,\cdot\,)_p=\frac{\p^p } {\p t^p}(\,\cdot\,)$.

Suppose $(g_{AB})_p=0$ at $t=0$ if $A,B$ are of the same type  for $p=0,1,\dots,k$, then  it is also true that $(g^{AB})_p=0$ if $A,B$ are of the same type  for $p=0,1,\dots,k$.
On the other hand, in the RHS of the above inequality, the derivative of each term with respect to $t$ is only up to order $k$, by the induction hypothesis,  $H_1(k+1)$ holds. Now
\bee
\frac{\p }{\p t}g_{AB}=-2R_{AB},
\eee
it is easy to see that $H_2(k+1)$ holds because $H_3(k)$ holds.

Since $$\frac{\partial^{k+1}}{\partial t^{k+1}} R_{\alpha\beta}=\sum\limits_{\tiny{\begin{array}{c}
m+n=k+1 \\
0\le m,n\le k+1
\end{array}}
}(g^{CD})_m(R_{\alpha CD\beta})_n,$$ and since that $H_1(k+1)$ and $H_2(k+1)$ hold, we conclude that $H_3(k+1)$ holds. Here we have used the symmetries of $R_{ABCD}$.

Since  \bee
\begin{split}
\frac{\partial^{k+1}}{\partial t^{k+1}} \Gamma^\alpha_{A\bar\b}=&
-\sum\limits_{\tiny{\begin{array}{c}
m+n=k \\
0\le m,n\le k
\end{array}}
}(g^{\a D})_m(R_{\bar\b D;A}+R_{A D;\bar\b}-R_{A \bar\b;D})_n\\
=&-\sum\limits_{\tiny{\begin{array}{c}
m+n=k \\
0\le m,n\le k
\end{array}}
}(g^{\a \bar\sigma})_m(R_{\bar\b \bar\sigma;A}+R_{A \bar\sigma;\bar\b}-R_{A \bar\b;\bar\sigma})_n,
\end{split}
\eee
 by the induction hypothesis. If $A=\bar\gamma$, then  each term on the RHS is zero by the induction hypothesis. If $A=\gamma$, then
\bee
 (R_{\bar\b\bar\sigma;\gamma})_n=(R_{\bar\b\bar\sigma,\gamma})_n-
 (\Gamma^E_{\gamma\bar\sigma}R_{E\bar\b})_n-(\Gamma^E_{\gamma\bar\b}R_{E\bar\sigma})_n, \eee so it vanishes because $n\le k$. On the other hand,
 \bee \begin{split}
& R_{\gamma \bar\sigma;\bar\b}-R_{\gamma \bar\b;\bar\sigma}\\
=&g^{CD}(R_{\gamma CD\bar\sigma;\bar\b}-R_{\gamma CD\bar\b;\bar\sigma})\\
=&g^{CD}(R_{\gamma CD\bar\sigma;\bar\b}+R_{\gamma C\bar\sigma  D ;\bar\b}+R_{\gamma C\bar\b\bar\sigma;D})\\
=&g^{CD}R_{\gamma C\bar\b\bar\sigma;D}.  \end{split} \eee
So
\bee
\lf(R_{\gamma \bar\sigma;\bar\b}-R_{\gamma \bar\b;\bar\sigma}\ri)_n=0
\eee
for $n\le k$ by the induction hypothesis.  Thus,
$$\frac{\partial^{k+1}}{\partial t^{k+1}} \Gamma^\alpha_{A\bar\b}=0
$$
at $t=0$. Since $\Gamma_{AB}^C=\Gamma_{BA}^C$ and $\ol{\Gamma_{AB}^C}=\Gamma_{\bar A\bar B}^{\bar C}$, it is easy to see that
$H_4(k+1)$ holds.

Next, \bee \begin{split} R_{AB\gamma\delta;E}=&R_{AB\gamma\delta,E}-\Gamma^G_{EA}R_{GB\gamma\delta}-\Gamma^G_{EB}R_{AG\gamma\delta}\\
&-\Gamma^G_{E\gamma}R_{ABG\delta}-\Gamma^G_{E\delta}R_{AB\gamma G}. \end{split} \eee
By $H_1(k+1)$, we have
$$
\frac{\p^{k+1}}{\p t^{k+1}}R_{AB\gamma\delta,E}=(\frac{\p^{k+1}}{\p t^{k+1}}R_{AB\gamma\delta})_{,E}=0.
$$
Since $H_1(i)$ and $H_4(i)$ are true for $0\le i\le k+1$,   $H_5(k+1)$ is true.
 Since $H_1(i)$,  $H_4(i)$ and $H_5(i)$  are true for $0\le i\le k+1$,   $H_6(k+1)$ is true.
 Finally   $H_6(i)$ is true for $0\le i\le k+1$ implies that   $H_7(k+1)$ holds. Therefore, $H(k+1)$ holds.

 \end{proof}

 Now we use the Uhlenbeck's trick to simplify the evolution equation of the complex curvature tensor. We pick an abstract vector bundle $V$ over $M$ which is isomorphic to $T_{\mathbb{C}}M$ and denote the isomorphism $u_0: V\to T_{\mathbb{C}}M$. And we take $\{e_A:=u^{-1}_0(\ppza)\}$ as a basis of $V$. We also consider a metric $h$ on $V$ by $h:=u^\ast_0g_0$. We let $u_0$ evolute by \bee \left\{
\begin{array}{l}
  \ppt u(t)=\Ric\circ u(t) \\
  u(0)=u_0
\end{array} \right. \eee In local coordinate, we have
\bee
 \left\{\begin{array}{l}
 \ppt u^{A}_{B}=g^{AC}R_{CD}u^{D}_{B}, \\
  u^{A}_{B}(0)=\delta^A_B
\end{array} \right.
  \eee

  Consider metric $h(t):=u^\ast(t)g(t)$ on $V$ for each $t\in[0,T]$. It is easy to see that $\ppt h(t)\equiv 0$ for all $t$, so $h(t)\equiv h$ for all $t$. We use $u(t)$ to pull the curvature tensor on $T_{\mathbb{C}}M$ back to $V$:
\bee
\widetilde{Rm}(e_A,e_B,e_C,e_D):=R(u(e_A),u(e_B),u(e_C),u(e_D)). \eee In local coordinate, we have \bee
\tilde{R}_{ABCD}=R_{EFGH}u^E_Au^F_Bu^G_Cu^H_D \eee on $M\times[0,T]$. One can also check \bee
\overline{h_{AB}}=h_{\bar A\bar B},\  \overline{\t R_{ABCD}}=\t R_{\bar A\bar B\bar C\bar D}. \eee

Define a connection on $V$ in the following: For any smooth section $\xi$ on $V$, $X\in T_{\mathbb{C}}M$,
\bee
D^t_X\xi=u^{-1}(\cd^t_X(u(\xi))).
\eee
 One can check $D^th=0$ and $D^tu=0$. We define $\triangle$ acting on any tensor on $V$ by \bee
\triangle:=g^{EF}D^t_ED^t_F.
\eee
Then by \eqref{e-1}, the evolution equation of $\t R$ is:
 \be \label{2}\begin{split}
\ppt \t R_{ABCD}=&\triangle \t R_{ABCD}-2h^{EF}h^{GH}R_{EABG}R_{FHCD}-2h^{EF}h^{GH}\t R_{EAGD}\t R_{FBHC}\\
&+2h^{EF}h^{GH}\t R_{EAGC}\t R_{FBHD} \end{split} \ee where $h^{AB}=(h^{-1})^{AB}$.

\begin{lma}\label{l-timederivative-2} With the above notations, we have
$$
\pptk \t R_{AB\gamma\delta}=0
$$
at $t=0$ for all $A, B$ and $\gamma,\delta$.

\end{lma}
\begin{proof} Note that we have:
 \bee \pptk \wt R_{AB\gamma\delta}=\sum\limits_{\tiny{\begin{array}{c}
 m+n+p+q+r=k \\
 0\le m,n,p,q,r\le k
 \end{array}}
}(u^E_A)_m(u^F_B)_n(u^G_\gamma)_p (u^H_\delta)_q (R_{EFGH})_r.
\eee
By Lemma \ref{l-timederivative-1}, in order to prove the lemma, it is sufficient to prove that $\pptk u^\a_{\bar\b}=0$ and $\pptk u^{\bar\a}_\b=0$ for all $k$ for all $\a, \b$ at $t=0$.

Recall that
\bee
 \left\{\begin{array}{l}
 \ppt u^{A}_{B}=g^{AC}R_{CD}u^{D}_{B}, \\
  u^{A}_{B}(0)=\delta^A_B.
\end{array} \right.
  \eee
  Hence $  u^\a_{\bar\b}=0$ and $ u^{\bar\a}_\b=0$. By induction, Lemma \ref{l-timederivative-1}, and the fact that $u^{A}_{B}(0)=\delta^A_B$, one can prove that  show $\pptk u^\a_{\bar\b}=0$ and $\pptk u^{\bar\a}_\b=0$ for all $k$. This completes the proof of the lemma.
\end{proof}

\begin{proof}[Proof of Theorem \ref{t-Kahler}] As in \cite{Shi1997},  define a smooth function $\varphi$
on $M\times[0,T]$ by
\be\label{e-phi} \begin{split}
\varphi(z,t)&=h^{\alpha\bar\xi}h^{\beta\bar\zeta}h^{\gamma\bar\sigma}h^{\delta\bar\eta}\t R_{\alpha\beta\gamma\delta}\t R_{\bar\xi\bar\zeta\bar\sigma\bar\eta}+h^{\bar\alpha\xi}h^{\bar\beta\zeta}h^{\gamma\bar\sigma}h^{\delta\bar\eta}\t R_{\bar\alpha\bar\beta\gamma\delta}\t R_{\xi\zeta\bar\sigma\bar\eta}\\
&+h^{\bar\alpha\xi}h^{\beta\bar\zeta}h^{\gamma\bar\sigma}h^{\delta\bar\eta}\t R_{\bar\alpha\beta\gamma\delta}\t R_{\xi\bar\zeta\bar\sigma\bar\eta}+h^{\alpha\bar\xi}h^{\bar\beta\zeta}h^{\gamma\bar\sigma}h^{\delta\bar\eta}\t R_{\alpha\bar\beta\gamma\delta}\t R_{\bar\xi\zeta\bar\sigma\bar\eta}
\end{split} \ee
One can check $\varphi$ is well-defined (independent of coordinate changes on $M$) and is nonnegative. The evolution equation of $\varphi$ is (See \cite{Shi1997}):
\be\label{7}
(\ppt-\triangle)\varphi=\t R_{CDEF}\ast\t R_{GH\alpha\beta}\ast\t R_{AB\gamma\delta}-2g^{EF}\t R_{AB\gamma\delta;E}\overline{\t R_{AB\gamma\delta}}_{;F}. \ee
As the real case,   define the norm of the complex curvature tensor by: \bee
|R_{ABCD}(t)|^2_{g(t)}=g^{AE}g^{BF}g^{CG}g^{DH}R_{ABCD}R_{EFGH}. \eee
Then we have \bee
|R_{ABCD}(t)|=|R_{ijkl}(t)|\leq \frac{a}{t} \eee  on $M\times[0,T]$ by assumption.  By the definition of $\wt R_{ABCD}$, we also have:
$$
|\wt R_{ABCD}(t)|=|R_{ABCD}(t)|\leq \frac{a}{t}.
$$

Combining with  \eqref{7}, we have
\bee
(\ppt-\triangle)\varphi \le \frac{C_1}t \varphi
\eee
on $M\times[0,T]$ for some constant $C_1$. Moreover,
$$
\varphi\le |\wt R_{ABCD}(t)|^2\le a^2/t^2.
$$
 On the other hand, by \eqref{e-phi}, Lemma \ref{l-timederivative-2} and the fact that $h$ is independent of $t$, we conclude that at $t=0$,
$$
\pptk\varphi=0,
$$
for all $k$. By Lemma \ref{l-maximum}, we conclude that $\varphi\equiv0$ on $M\times[0,T]$. As in \cite{Shi1997}, we conclude that $g(t)$ is \K for all $t>0$.

\end{proof}

\begin{cor}\label{c-CW-solution} Let $(M^n,g_0)$ be a complete noncompact \K manifold with complex dimension $n$ and with nonnegative complex sectional curvature. Suppose

\bee
\inf_{p\in M}\{V_p(1)|\ p\in M\}=v_0>0.
\eee
where $V_p(1)$ is the volume of the geodesic ball with radius 1 and center at $p$ with respect to $g_0$.
Then there is $T>0$ depending only on $n, v_0$ such that the \KR flow has a complete solution on $M\times[0,T]$ such that $g(t)$ has nonnegative complex sectional curvature. Moreover  the curvature satisfies:
$$
|\Rm(g(t))|_{g(t)}\le \frac ct
$$
where $c$ is a constant depending only on $n, v_0$ with initial data $g(0)=g_0$.
\end{cor}
\begin{proof}
The corollary follows immediately from the result of Cabezas-Rivas and Wilking\cite{Cabezas-RivasWilking2011}, and Theorem \ref{t-Kahler}.
\end{proof}

\begin{cor}\label{c-CW-maxvol}
Let $(M^n,g_0)$ be a complete noncompact \K manifold with complex dimension $n$ and with nonnegative complex sectional curvature. Suppose $M^n$ has maximum volume growth. Then $M^n$ is biholomorphic to $\C^n$.
\end{cor}
\begin{proof} By volume comparison, we have
\bee
\inf_{p\in M}\{V_p(1)|\ p\in M\}=v_0>0.
\eee
Let $g(t)$ be the solution of \KR flow on $M\times[0,T]$ obtained as in   Corollary \ref{c-CW-solution}. Then for all $t>0$, $g(t)$ has nonnegative complex sectional curvature and the curvature of $g(t)$ is bounded. We want to prove that $g(t)$ has maximum volume growth.

Let $p\in M$ and let $r>0$ be fixed. Let $\tilde g(s)=r^{-2}g(r^2 s)$, $0\le s\le r^{-2}T$. Then $\tilde g(s)$ is a solution to the \KR flow with initial data $\tilde g(0)=r^{-2}g_0$. Since the sectional curvature of $\tilde g(s)$ is nonnegative, as in \cite{Cabezas-RivasWilking2011}, using a result of \cite{Petrunin}, one can prove that:
\bee
V_p(\tilde g(s),1)-V_p(r^{-2}g_0,1)=V_p(\tilde g(s),1)-V_p(\tilde g(0),1)\ge -c_ns
\eee
where $V_p(h,1)$   denotes the volume of the geodesic ball with radius 1 and center at $p$ with respect to $h$ and $c_n$ is a positive constant depending only on $n$.
Now
\bee
\begin{split}
V_p(r^{-2}g_0,1)=&\frac{V_p(g_0,r)}{r^{2n}}\ge v_0>0
\end{split}
\eee
because $g_0$ has maximum volume growth. Hence there is $r_0>0$ such that if $r\ge r_0$, then
\bee
r^{-2n}V_p(g(r^2s),r)=V_p(\tilde g(s),1)\ge C_1
\eee
for some constant independent of $s$ and $r$ for all $0\le s\le r^{-2}T$.   Fix $t_0>0$, and let $s$ be such that $r^2 s=t_0$. Then $s\le r^{-2}T$. So we have
\bee
r^{-2n}V_p(g(t_0),r)\ge C_1
\eee
if $r$ is large enough. That is, $g(t_0)$ has maximum volume growth. By \cite{ChauTam2008}, we conclude that $M$ is biholomorphic to $\C^n$.

\end{proof}

 \section{Preservation of non-negativity of holomorphic bisectional curvature}\label{s-bisectional}

Let $(M^n,g_0)$ be a complete noncompact \K manifold with complex dimension $n$. We want to study the preservation of non-negativity of holomorphic bisectional curvature under \KR flow, without assuming the curvature is bounded in space and time.

Let us first define a quadratic form for any $(0,4)$-tensor $T$ on $T_\mathbb{C}M$ with a metric $g$ by

\bee \begin{split} Q(T)(X,\bar X,Y,\bar Y):=&\sum\limits_{\mu,\nu=1}^n(|T_{X\bar\mu\nu\bar Y}|^2-|T_{X\bar\mu Y\bar\nu}|^2+T_{X\bar X\nu\bar \mu}T_{\mu\bar \nu Y\bar Y})\\
&-\sum\limits_{\mu=1}^n\textbf{Re}(T_{X\bar\mu}T_{\mu\bar XY\bar Y }+T_{Y\bar\mu}T_{X\bar X\mu\bar Y})\end{split}
\eee
for all $X,Y\in T^{1,0}_\mathbb{C}M$, where $T_{\a\bb\gamma\bar\delta}=T(e_\a,\bar e_\b,e_\gamma,\bar e_\delta)$,
$T_{\a\bar \b}=g^{\gamma\bar \delta}T_{\a\bb\gamma\bar\delta}$ and $\{e_1,\dots,e_n\}$ is a unitary frame with respect to the metric of $g$, $T_{X\bar \mu\nu\bar Y}=T(X,\bar e_\mu,e_\nu,\bar Y)$ etc. Here $T$ is a tensor has the following properties:
\bee
\ol{T(  X,Y,Z,W)}=T(\bar  X,\bar Y,\bar Z,\bar W);
\eee
\bee
T(X,Y,Z,W)=T(Z,W,X,Y)=T(X,W,Z,Y)=T(Y,X,W,Z).
\eee

Let $g(t)$ be a solution of the \KR flow:
\bee
\ppt g_{\a\bb}=-R_{\a\bb}.
\eee
Recall the evolution equation for holomorphic bisectional curvature: (See \cite[Corollary 2.82]{Chow07})

\bee (\ppt-\triangle)R(X,\bar X,Y,\bar Y)=Q(R)(X,\bar X,Y,\bar Y) \eee for all $X,Y\in T^{1,0}_\mathbb{C}M$. Here $\triangle$ is with respect to $g(t)$.

Next define a $(0,4)$-tensor $B$ on $T_\mathbb{C}M$ (with a metric $g$) by:
\bee
B(E,F,G,H)=g(E,F)g(G,H)+g(E,H)g(F,G) \eee
for all $E,F,G,H\in T_\mathbb{C}M$.

\begin{lma}\label{qb} In the above notation, $Q(B)(X,\bar X,Y,\bar Y)\leq 0$ for all $X,Y\in T^{1,0}_\mathbb{C}M$.
\begin{proof}
\bee
\begin{split} Q(B)(X,\bar X,Y,\bar Y)=&\sum\limits_{\mu,\nu=1}^n(|B_{X\bar\mu\nu\bar Y}|^2-|B_{X\bar\mu Y\bar\nu}|^2+B_{X\bar X\nu\bar \mu}B_{\mu\bar \nu Y\bar Y})\\
&-\sum\limits_{\mu=1}^n\textbf{Re}(B_{X\bar\mu}B_{\mu\bar XY\bar Y }+B_{Y\bar\mu}B_{X\bar X\mu\bar Y})\end{split}
\eee
Let $\{e_1,\dots,e_n\}$  be a unitary frame. $X=\sum_{\mu=1}^nX^\mu e_\mu,$ $\sum_{\mu=1}^n Y=Y^\mu e_\mu$.

We compute it term by term: \bee\begin{split}
\sum\limits_{\mu,\nu=1}^n|B_{X\bar\mu\nu\bar Y}|^2=&\sum\limits_{\mu,\nu=1}^n(g_{X\bar\mu}g_{\nu\bar Y}+g_{X\bar Y}g_{\bar\mu\nu})\cdot(g_{\bar X\mu}g_{\bar\nu Y}+g_{\bar X Y}g_{\mu\bar\nu})\\
=&\sum\limits_{\mu,\nu=1}^n(X^\mu\bar Y^\nu+g_{X\bar Y}g_{\bar\mu\nu})\cdot(\bar X^\mu Y^\nu+g_{\bar X Y}g_{\mu\bar\nu})\\
=&|X|^2|Y|^2+(n+2)|g(\bar X, Y)|^2. \end{split} \eee

\bee\begin{split}
\sum\limits_{\mu,\nu=1}^n|B_{X\bar\mu Y\bar\nu}|^2=&\sum\limits_{\mu,\nu=1}^n(g_{X\bar\mu}g_{\bar\nu Y}+g_{X\bar \nu}g_{\bar\mu Y})\cdot(g_{\bar X\mu}g_{\nu \bar Y}+g_{\bar X \nu}g_{\mu\bar Y})\\
=&\sum\limits_{\mu,\nu=1}^n(X^\mu Y^\nu+X^\nu Y^\nu)\cdot(\bar X^\mu \bar Y^\nu+\bar X^\nu\bar Y^\mu)\\
=&2|X|^2|Y|^2+2|g(\bar X, Y)|^2. \end{split} \eee

\bee\begin{split}
\sum\limits_{\mu,\nu=1}^n B_{X\bar X \nu\bar\mu}B_{\mu\bar\nu Y\bar Y}=&(g_{X\bar X}g_{\nu\bar\mu}+g_{X\bar\nu}g_{\bar X\nu})\cdot(g_{\mu\bar\nu}g_{Y\bar Y}+g_{\mu\bar Y}g_{\bar\nu Y})\\
=&(|X|^2g_{\nu\bar\mu}+X^\mu\bar X^\nu)\cdot(g_{\mu\bar\nu}|Y|^2+\bar Y^\mu Y^\nu)\\
=&(n+2)|X|^2|Y|^2+|g(\bar X, Y)|^2. \end{split} \eee

\bee\begin{split}
\sum\limits_{\mu=1}^n B_{X\bar\mu}B_{\mu\bar XY\bar Y}=&\sum\limits_{\mu=1}^n g^{k\bar l}B_{X\bar\mu k\bar l}B_{\mu\bar XY\bar Y}\\
=&\sum\limits_{\mu,\nu=1}^n B_{X\bar\mu \nu\bar\nu}B_{\mu\bar XY\bar Y}\\
=&\sum\limits_{\mu,\nu=1}^n (g_{X\bar\mu}g_{\nu\bar\nu}+g_{X\bar\nu}g_{\bar\mu\nu})\cdot(g_{\bar X\mu}g_{\bar YY}+g_{\bar X Y}g_{\mu\bar Y})\\
=&\sum\limits_{\mu,\nu=1}^n (X^\mu g_{\nu\bar\nu}+X^\nu g_{\bar\mu\nu})\cdot(\bar X^\mu g_{Y\bar Y}+\bar Y^\mu g_{\bar XY})\\
=&(n+1)|X|^2|Y|^2+(n+1)|g(\bar X, Y)|^2. \end{split} \eee

Similarly, we have

\bee
\sum\limits_{\mu=1}^n B_{Y\bar\mu}B_{X\bar X\mu\bar Y}=(n+1)|X|^2|Y|^2+(n+1)|g(\bar X, Y)|^2 \eee

Therefore, \bee Q(B)(X,\bar X,Y,\bar Y)=-(n+1)(|X|^2|Y|^2+|g(\bar X, Y)|^2)\leq 0. \eee
\end{proof}
\end{lma}

We are ready to prove Theorem \ref{t-intro-2}:

\begin{thm}\label{nonnegative-bi} There is $0< a(n)<1$ depending only on $n$ such that if $g(t)$ is a complete solution of  \KRF on $M\times[0,T]$ with $|\Rm(g(t))|_{g(t)}\leq \frac{a}{t}$, where $M$ is an $n$-dimensional non-compact complex manifold. If $g(0)$ has nonnegative holomorphic bisectional curvature, then so does $g(t)$ for all $t\in[0,T]$.
\begin{proof} The theorem   is  known to be true  if the curvature is uniformly bounded on space and time \cite{Shi1997}. Since $g(t)$ has bounded curvature on $M\times[\tau,T]$ for all $\tau>0$, it is sufficient to prove that $g(t)$ has nonnegative bisectional curvature on $M\times[0,\tau]$ for some $\tau>0$. Hence we may assume that $T\le 1$.

 In the following,  lower case $c_1,c_2,\cdots$ will denote  constants  depending only on $n$.

Since $g(T)$ has bound curvature $\frac a T$ and is complete, as in   Lemma \ref{l-cutoff},  a smooth function $\rho$ defined on $M$ such that
\bee
(1+d_T(x,p))  \leq \rho(x)\leq D_1 (1+d_T(x,p))
\eee
\be\label{e-rho}
|\bar\nabla\rho|+|\bar\nabla^2\rho|\leq D_1,
\ee
for some constant $D_1$ depending only on $n$ and $g_T$, where $d_T(x,p)$ is the distance function with respect to $g_T$ from a fixed point $p\in M$, where $\bar\n$ is the covariant derivative with respect to $g_T$

Suppose $|\Rm(g(t))|_{g(t)}\le a/t$, where $a$ is to be determined later depending only on $n$. By Lemma \ref{l-cutoff}, we have
\be\label{e-D-rho}
|\nabla \rho|\le D_2 t^{-c_1a}
\ee
for some constant $D_2$ depending only on $n, g_T$. Here and below, $\nabla $ is the covariant derivative of $g(t)$ and hence is time dependent. We may get a better estimate for $\Delta\rho=\Delta_{g(t)}\rho$ than that in Lemma \ref{l-cutoff}. Choose a normal coordinate with respect to $g(T)$ which also diagonalizes $g(t)$ with eigenvalues $\lambda_\a$. Then

\be\label{e-2}
|\triangle\rho|=|g^{\abb} \rho_{\abb}|=\sum_{\a=1}^n\lambda_\a^{-1}|\bar\nabla^2\rho|\le D_3 t^{-c_2a}, \ee
by Lemma \ref{l-connection}.

Let $\phi$ be a smooth cut-off function   from $\mathbb{R}$ to $[0,1]$ such that \bee
  \phi(x)=\left\{\begin{array}{cc}
                    1, & x\leq 1  \\
                    0, & x\geq 2
                  \end{array} \right.
\eee and $|  \phi'|+|  \phi''|\leq D'$, $\phi'\le0$. Let $\Phi=\phi^m$, where $m>4$ is an integer to be determined later. Then
\bee
0\ge \Phi'\ge -D(m)\Phi^q;\ \  |\Phi''|\le D(m)\Phi^q
\eee
 for some positive constant $D(m)$ depending only on $D'$ and $m$, where $q=1-\frac2m$.

Let  $\Psi(x)=\Phi(\frac{\rho(x)}{r})$ on $M$ for $r\geq 1$. Note that $\Psi$ depends on $r$.

Then we have \be\label{e-cutoff-1}\begin{split}
|\n\Psi|\le&\frac{1}{r}D(m)\Psi^q|\n\rho|\\
 \leq&\frac{D_4 }{r}\Psi^qt^{-c_1a} \end{split}\ee
by \eqref{e-D-rho}, and
\be\label{e-cutoff-2}
 |\triangle\Psi|\le\frac{1}{r^2}|\Phi''||\n\rho|^2+\frac{1}{r}|\Phi'\triangle\rho|\le \frac{D_4}{r}\Psi^qt^{-c_2a}
\ee
by \eqref{e-2},
 where $D_4$ is a constant depending only on $n,g_T,m$.

For any  $\varepsilon>0$, we define a tensor $A$ on $M\times(0,T]$: For vectors $X, Y,Z, W\in T_\C(M)$,
\bee
A(X,Y,Z,W) =t^{-\frac{1}{2}}\Psi(x)R(X,Y,Z,W) +\varepsilon B(X,Y,Z,W)
\eee
where $R$ is the curvature tensor of $g(t)$ and $B$ is evaluated with respect to $g(t)$.

Define the following  function on $M\times(0,T]$:
 \bee
H(x,t)=\inf\{A_{X\bar XY\bar Y}(x,t)| |X|_t=|Y|_t=1, X,Y\in T^{(1,0)}_xM\}. \eee Here $|\cdot|_t$ is the norm with respect to $g(t)$.

To show the theorem, it suffices to show for all $r>>1$, $H(x,t)\geq 0$ for all $x$ and  for all $t>0$.  Note that $t^\frac12 H(x,t)$ is a continuous function.  Since $\Psi$ has compact support,
and $B(X,\bar X,Y,\bar Y)\ge1$ for all $|X|_t=|Y|_t=1$,   there is a compact set $K\in M$ such that
\bee H(x,t)>0 \eee on $(M\setminus K)\times(0,T]$.
On the other hand,
we   claim that there is $T_0>0$ such that
\bee t^\frac12 H(x,t)>0
\eee on $K\times(0,T_0)$. Let $\{e_1,e_2,\dots,e_n\}$ be a unitary frame near a compact neighborhood $U$  of a  point $x_0\in K$ with respect to $g_0$. Then at each point $x\in U$,

$$
R_{\a\bb\gamma\bar\delta}(x,t)=R_{\a\bb\gamma\bar\delta}(x,0)+tE
$$
where $|E|$ is uniformly bounded on $U\times[0,T]$. Since $g(t)$ is uniformly equivalent to $g(0)$ on $U$,
for any $X, Y\in T^{1,0}_x(M)$ for $(x,t)\in U\times[0,T]$,
$$
R(X,\bar X,Y,\bar Y)\ge  -D_5t|X|^2_{0}|Y|^2_{0}
$$
for some constant $D_5>0$ where  we have used the fact that $g_0$ has nonnegative holomorphic bisectional curvature. Since $g(t)$ and $g_0$ are uniformly equivalent in $K$, and $K$ is compact, we conclude that
$$
R(X,\bar X,Y,\bar Y)\ge  -D_6t
$$
on $K\times[0,T]$ for some constant $D_6$ for all $X, Y\in T^{1,0}_x(M)$ with $|X|_t=|Y|_t=1$. Since $t^\frac12 B(X,\bar X,Y,\bar Y)\ge t^\frac12$, it is easy to see the claim is true. To summarize, we have proved that there is a compact set $K$ and there is $T_0>0$, such that $H(x,t)>0$ on $M\setminus K\times(0,T]$ and $K\times(0,T_0)$.

Suppose $H(x,t)<0$ for some $t>0$. Then $t^\frac12 H(x,t)<0$ for some $t>0$.   We must have $x\in K$ and $t\ge T_0$. Hence we can find $x_0\in K$, $t_0\ge T_0$ and a neighborhood $V$ of $x_0$ such that $H(x_0,t_0)=0$, $H(x,t)\ge 0$ for $x\in V$, $t\le t_0$. This implies that there exist $X_0,Y_0\in T^{(1,0)}_{x_0}M$ with norm $|X_0|_{g(t_0)}=|Y_0|_{g(t_0)}=1$
such that
\bee
A_{X_0\bar X_0Y_0\bar Y_0}(x_0,t_0)=0. \eee Then we extend $X_0,Y_0$ near $x_0$ by parallel translation with respect to $g(t_0)$ to vector fields  $\wt X_0,\wt Y_0$ such that they are independent of time and $$\Delta_{g(t_0)}\wt X_0=\Delta_{g(t_0)}\wt Y_0=0,$$ at $x_0$.

Denote $h(x,t):=A_{\wt X_0\bar{\wt X_0}\wt Y_0\bar{\wt Y_0}}(x,t)$. At $(x_0,t_0)$, we have  $h(x_0,t_0)=0$ and   $h(x,t)\ge0$ for $x\in V$, $t\le t_0$ by the definition of $A$.

Hence at  $(x_0,t_0)$,
 \be\label{e13} \begin{split}
0\ge&(\ppt-\triangle)h\\
=&t_0^{-\frac{1}{2}}\Psi\lf( \heat  R\ri)(  X_0,\bar{  X_0},  Y_0,\bar{  Y_0}))-t_0^{-\frac12}R (  X_0,\bar{ X_0},  Y_0,\bar{  Y_0})\Delta\Psi\\
&-2 t_0^{-\frac12}\la\nabla R (\wt  X_0,\bar{\wt X_0},  \wt Y_0,\bar{\wt  Y_0}),\nabla \Psi\ra-\frac12 t_0^{-\frac32}\Psi R (  X_0,\bar{ X_0},  Y_0,\bar{  Y_0})-\varepsilon(\Delta B)(  X_0,\bar{ X_0},  Y_0,\bar{  Y_0})\\
&+\varepsilon (-\Ric( X_0,\bar{X_0})-\Ric(Y_0,\bar Y_0)- \Ric(X_0,\bar Y_0)g(X_0,\bar Y_0)-\Ric(\bar X_0,Y_0)g(X_0,\bar Y_0))\\
\ge & t_0^{-\frac{1}{2}}\Psi Q(R)(  X_0,\bar{  X_0},  Y_0,\bar{  Y_0}))-D_4r^{-1}t_0^{-\frac12-c_2a}\Psi^q|R(  X_0,\bar{  X_0},  Y_0,\bar{  Y_0}))| \\
&-\frac12 t_0^{-\frac32}\Psi R (  X_0,\bar{ X_0},  Y_0,\bar{  Y_0})-c_3\varepsilon  at_0^{-1}-2 t_0^{-\frac12}\la\nabla R ( \wt X_0,\bar{\wt X_0}, \wt Y_0,\bar{\wt  Y_0}),\nabla \Psi\ra,
\end{split}\ee
where we have used \eqref{e-cutoff-2} and the fact that $\Delta B=0$. On the other hand, at $(x_0,t_0)$
\bee
\begin{split}
0=&\nabla h\\
=&t_0^{-\frac12}\n \lf(R(  \wt X_0,\bar{\wt X_0}, \wt Y_0,\bar{\wt  Y_0})\Psi\ri)\\
=&t_0^{-\frac12}\lf[\Psi\n R (  \wt X_0,\bar{\wt X_0}, \wt Y_0,\bar{\wt  Y_0})+R (    X_0,\bar{  X_0},   Y_0,\bar{   Y_0})\n\Psi\ri]
\end{split}
\eee
where we have used the fact that $\n g=0$ and $\n\wt X_0=\n\wt Y_0=0$ at $(x_0,t_0)$. Hence \eqref{e13} implies
 \be\label{e13-1}   \begin{split}
0\ge& t_0^{-\frac{1}{2}}\Psi Q(R)(  X_0,\bar{  X_0},  Y_0,\bar{  Y_0}))-D_7r^{-1}t_0^{-\frac12-c_4a}|R(  X_0,\bar{  X_0},  Y_0,\bar{  Y_0})|\lf(\Psi^q+\Psi^{2q-1}\ri) \\
&-\frac12 t_0^{-\frac32}\Psi R (  X_0,\bar{ X_0},  Y_0,\bar{  Y_0})-c_3\varepsilon at_0^{-1}
\end{split}
\ee
where we have used \eqref{e-cutoff-1}, where $D_7>0$ is a constant depending only on $g_T , n, m$.  On the other hand, by the null-vector condition \cite[Proposition 1.1]{Mok} (see also \cite{Bando1984}), we have
$$
Q(A)(X_0,\bar X_0,Y_0,\bar Y_0)\ge 0
$$
By a direct computation, one can see that
\bee
 Q(A)=t_0^{-1}\Psi^2 Q(R)+\varepsilon^2 Q(B)+t_0^{-\frac{1}{2}}\Psi\varepsilon R\ast B, \eee
and we have
\bee
\begin{split}
0\le&t_0^{-1}\Psi^2 Q(R) (X_0,\bar X_0,Y_0,\bar Y_0)+\varepsilon^2 Q(B)(X_0,\bar X_0,Y_0,\bar Y_0)+c_5\varepsilon\Psi at_0^{-\frac32}\\
\le&t_0^{-1}\Psi^2 Q(R) (X_0,\bar X_0,Y_0,\bar Y_0) +c_5\varepsilon \Psi at_0^{-\frac32}
\end{split}
\eee
where we have used Lemma \ref{qb} and $c_5$ is a constant depending only on $n$. That is
\be\label{e-null}
\begin{split}
 \Psi  Q(R) (X_0,\bar X_0,Y_0,\bar Y_0)\ge -c_5\varepsilon   at_0^{-\frac12}
\end{split}
\ee
where we have used the fact that $h(x_0,t_0)=0$ which implies $\Psi(x_0,t_0)>0$.

Combining this with \eqref{e13-1}, we have
\be\label{e13-2}   \begin{split}
0\ge& -(c_3+c_5)\varepsilon  at_0^{-1} -D_7r^{-1}t_0^{-\frac12-c_4a}|R(  X_0,\bar{  X_0},  Y_0,\bar{  Y_0})|\lf(\Psi^q+\Psi^{2q-1}\ri) \\
&-\frac12 t_0^{-\frac32}\Psi R (  X_0,\bar{ X_0},  Y_0,\bar{  Y_0}),
\end{split}
\ee

Since  $h(x_0,t_0)=0$, we also have
$$
\Psi(x_0,t_0)R(X_0,\bar X_0,Y_0,\bar Y_0)=- t_0^\frac12\varepsilon B(X_0,\bar X_0,Y_0,\bar Y_0).
$$
Hence at $(x_0,t_0)$, \eqref{e13-2} implies, if $0<a<1$, then
\bee   \begin{split}
0\ge& -(c_3+c_5)\varepsilon  a    -2D_7 r^{-1}t_0^{\frac12 -c_4a}  |R(  X_0,\bar{  X_0},  Y_0,\bar{  Y_0})| \Psi^{2q-1}  -\frac12 t_0^{-\frac12}  \Psi  R (  X_0,\bar{ X_0},  Y_0,\bar{  Y_0}) \\
\ge&-(c_3+c_5)\varepsilon  a    -2D_7 r^{-1}t_0^{\frac12 -c_4a}  |R(  X_0,\bar{  X_0},  Y_0,\bar{  Y_0})|^{2(1-q)}\,|\varepsilon t_0^\frac12 B(  X_0,\bar{  X_0},  Y_0,\bar{  Y_0})|^{2q-1}\\
&+\frac12 \varepsilon  B(  X_0,\bar{  X_0},  Y_0,\bar{  Y_0})\\
\ge&-(c_3+c_5)\varepsilon  a -D_8r^{-1}\varepsilon^{2q-1}t_0^\a+\frac12 \varepsilon
\end{split}
\eee
because $0\le \Psi\le 1$, $q=1-\frac 2m<1$, $m>4$ where $D_8>0$ are constants depending only on $g_T, n, m$.  Here
$$
\a=\frac12 -c_4a-2(1-q)+\frac12(2q-1)= 3q-c_4a-2.
$$
Hence if $c_4a<\frac12$ and $a<1$, then $a$ depends only on $n$ and $3q-c_4a-2>0$, provided $m$ is large enough. If $a, m$ are chosen satisfying these conditions, then we have
$$
0\ge  -(c_3+c_5)\varepsilon  a-D_8r^{-1}\varepsilon^{2q-1}+\frac12\varepsilon.
$$
If $a$ also satisfies $a(c_3+c_5)<\frac12$, then we have a contradiction if $r$ is large enough. Hence if
$$
0<a<\min\{1, \frac12 c_4^{-1},\frac12 (c_3+c_5)^{-1}\},
$$
then $g(t)$ will have nonnegative holomorphic bisectional curvature. This completes the proof of the theorem.

\end{proof}

\end{thm}

As an application, we have the following:

\begin{cor}\label{Xu} Let   $(M^n,g_0)$ is a complete noncompact \K manifold with complex dimension $n\ge2$ with nonnegative holomorphic bisectional curvature with maximum volume growth. Suppose there is $r_0>0$ and there is $C>0$ such that
$$
\lf(\aint_{B_x(r_0)}|\Rm|^p\ri)^\frac1p\le C
$$
for some $p>n$ for all $x\in M$. Then $M$ is biholomorphic to $\C^n$.
\end{cor}
\begin{proof} By \cite{Xu2013}, the Ricci flow with initial data $g_0$ has short time solution $g(t)$ so that the curvature has the following bound:
$$
|\Rm(g(t))|\le C t^{-\frac{n}{p}}
$$
for some constant $C$. Since $\frac np<1$,  by Theorems \ref{t-Kahler}and \ref{nonnegative-bi} $g(t)$ is \K and has bounded nonnegative bisectional curvature for $t>0$. Since $\frac np<1$ it is easy to see that $g(t)$ is uniformly equivalent to $g_0$. Hence $g(t)$ also has maximum volume growth. By \cite{ChauTam2008}, $M$ is biholomorphic to $\C^n$.

\end{proof}

\section{producing  \KR flow through $h$-flow}\label{s-hflow}

We want to produce solutions to \KR flow using the solutions of the so-called $h$-flow by M. Simon \cite{Simon2002}. Let us recall the set up and some results in \cite{Simon2002}. Let $ M^n $ be a smooth manifold, and let $g$ and $h$ be two Riemannian metrics on $M$. For a constant $\delta>1$, $h$ is said to be   $\delta$ close to $g$ if
 $$
 \delta^{-1}h\le g\le \delta h.
 $$
 Let $g(t)$ be a smooth family of metrics on $M\times[0,T]$, $T>0$. $g(t)$ is said to be a solution to the $h$-flow, if $g(t)$ satisfies
 following DeTurck flow, see \cite{Shi1989,Simon2002}:
 \be\label{e-hflow}
 \ppt g_{ij}=-2\Ric_{ij}+\n_iV_j+\n_jV_i,
 \ee
  where \bee
 V_i=g_{ij}g^{kl}(\Gamma^{j}_{kl}-{}^h\Gamma^{j}_{kl}),
 \eee
 and $\Gamma_{kl}^i$, ${}^h\Gamma_{kl}^i$ are the Christoffel symbols of $g(t)$ and $h$ respectively, and $\nabla$ is the covariant derivative with respect to $g(t)$. One can rewrite \eqref{e-hflow} in the following way which shows that it is a strictly parabolic system:
\bee\begin{split}
 \ppt g_{ij}=&g^{\a\b}\wt\n_\a\wt\n_\b g_{ij}-g^{\a\b}g_{ip}h^{pq}\wt\Rm_{j\a q\b}-g^{\a\b}g_{jp}h^{pq}\wt\Rm_{i\a q\b}\\
 &+\frac{1}{2}g^{\a\b}g^{pq}(\wt\n_i g_{p\a}\cdot\wt\n_j g_{q\b}+2\wt\n_\a g_{jp}\cdot\wt\n_q g_{i\b}-2\wt\n_\a g_{jp}\cdot\wt\n_\b g_{iq}\\
 &-2\wt\n_j g_{\a p}\cdot\wt\n_\b g_{iq}-2\wt\n_i g_{\a p}\cdot\wt\n_\b g_{jq}), \end{split} \eee
 where $\wt\n $ is covariant derivative with respect to $h$.

  In order to emphasis the background metric $h$, we  call it $h$-flow as in \cite{Simon2002}.
 We are only interested in the case that $M$ is noncompact and $g$ is complete.

 In \cite{Simon2002}, Simon obtained the following:
  \begin{thm}\label{t-Simon}[Simon]
  There is a $\e=\e(n)>0$ depending only on   $n$ such that if $(M^n,g_0)$ is a smooth $n$-dimensional complete noncompact manifold such that there is a smooth Riemannian metric $h$ with $|\nabla^i\Rm(h)|\le k_i$ for all $i$ and is $(1+\e(n))$ close to $g_0$, then the $h$-flow \eqref{e-hflow} has a  smooth  solution on $M\times[0,T]$ for some $T>0$ with $T$ depending only on $n, k_0$ such that $g(t)\to g_0$ as $t\to 0$ uniformly on compact sets and such that
  $$
  \sup_{x\in M}|\nabla^i g(t)|^2\le \frac{C_i}{t^i}
  $$
  for all $i$, where $C_i$ depends only on $n, k_0,\dots,k_i$. Moreover, $h$ is $(1+2\e)$ close to $g(t)$ for all $t$. Here and in the following $\n$ and $|\cdot|$ are with respect to $h$.
  \end{thm}

  From this and  using Theorem \ref{t-Kahler}, one can construct solution $g(t)$ to the \KR flow if $g_0$ is \K. Moreover, the curvature of $g(t)$ is bounded by $C/t$. However, we motivated by the uniformization conjecture of Yau \cite{Y}, we also want to prove that if $g_0$ has nonnegative holomorphic bisectional curvature, then one can construct solution to the \KR flow so that $g(t)$ also has nonnegative holomorphic bisectional curvature. To achieve this goal, we want to apply Theorem \ref{nonnegative-bi}. Therefore, we want to show that $C$ in the curvature bound  $C/t$ above is small provided $\e$ is small. We need   more refined estimates of $|\nabla g|$ and $|\nabla^2g|$. We proceed as in \cite{Shi1989,Simon2002}

Recall the evolution equations $|\nabla^p g|^2$, $p=1, 2$. Let $\wt\Rm$ be the curvature tensor of $h$ and let   $$\square =\ppt-g^{ij}\n_i\n_j.$$  Then (see \cite{Simon2002}):

 \be\label{boxdg} \begin{split}
\square |\n g|^2=&-2g^{kl}\n_{k}\n g_{ij}\cdot\n_{l}\n g_{ij}\\
&+\wt\Rm\ast g^{-1}\ast\n g\ast \n g+\wt\Rm\ast g^{-1}\ast g^{-1}\ast g\ast\n g\ast\n g\\
&+g^{-1}\ast g\ast\n\wt\Rm\ast\n g+g^{-1}\ast g^{-1}\ast \n g\ast\n g\ast\n^2 g\\
&+g^{-1}\ast g^{-1}\ast g^{-1}\ast \n g\ast \n g\ast \n g\ast\n g,
\end{split} \ee and

 \be\label{boxddg}\begin{split}
\square(|\n^2 g|^2)=&-2g^{ij}\n_i(\n^2g)\n_j(\n^2g)\\
&+\sum\limits_{i+j+k=2, 0\leq i,j,k\leq 2}\n^i g^{-1}\ast\n^jg\ast\n^k\wt\Rm\ast\n^2g\\
&+\sum\limits_{i+j+k+l=4, 0\leq i,j,k,l\leq 3}\n^i g^{-1}\ast\n^jg^{-1}\ast\n^kg\ast\n^lg\ast\n^2g.
\end{split} \ee
Here for tensors   $S_1*S_2$ denotes some trace with respect to $h$ of tensors $S_1, S_2$. The total numbers of terms on the R.H.S. of each equation depend only on $n$.

 \begin{lma}\label{l-d1-estimate} Let $(M^n,h)$ be a complete noncompact Riemannian manifold such that $|\Rm(h)|\le k_0$, and $|\nabla\Rm(h)|\le k_1$ with $k_0+k_1\leq 1$. For any $\a>0$ there is a constant $b(n,\a)>0$ depending only on $n$ and $\a$ such that $e^{2b}\le 1+\e(n)$ where $\e(n)$ is the constant in Theorem \ref{t-Simon}, and if $g(t)$ is the solution of the $h$-flow on $M\times [0,T]$, $T\leq 1$ obtained in Theorem \ref{t-Simon} with $g(0)=g_0$ satisfies $e^{-b}h\le g_0\le e^bh$, then there is a $T_1(n,\a)>0$ depending only on $n,\a$ such that
 \bee
 |\nabla g(t)|^2\le \frac \a t
 \eee
 for all $t\in (0,T_1]$.
 \end{lma}
 \begin{proof} By a $\mathcal{C}^0$-estimate of $h$-flow (See \cite[Theorem 2.5]{Shi1989} or \cite[Theorem 2.3]{Simon2002}), the constant  $\e(n)>0$ in Theorem \ref{t-Simon} can be chosen such that if $h$ is $e^b$ close to $g_0$ with $e^b\le 1+\e(n)$, then the solution $g(t)$ in Theorem \ref{t-Simon} is defined on $M\times[0,T_2]$ for some $T\ge T_2=T_2(n)>0$ (note that we assume $k_0+k_1\le 1$. Moreover
 \be\label{e-fair}
 e^{-2b}h\leq g(t)\leq e^{2b}h
  \ee
   for all $t\in[0,T_2]$.

   Let $f_0=|g|$, $f_1=|\n g|$ and $f_2=|\n^2 g|$.
  First choose $b>0$ such that:

  ({\bf c1}) $e^{2b}\le 2$ and $e^b\le 1+\e(n)$.

  Then we have $f_0\leq c_1$. Here and in the following, lower case $c_i$ will denote positive constants depending only on $n$.

 Using \eqref{e-fair}, we estimate terms of R.H.S. of \eqref{boxdg} in the following:
\bee \begin{split}
&g^{\alpha\beta}\n_{\alpha}\n g_{ij}\cdot\n_{\beta}\n g_{ij}\geq \frac{1}{2}f_2^2 \\
&\wt\Rm\ast g^{-1}\ast\n g\ast\n g\leq c_2k_0f_1^2\\
&\wt\Rm\ast g^{-1}\ast g^{-1}\ast g\ast\n g\ast\n g\leq c_2k_0f_1^2\\
&g^{-1}\ast g\ast\n\wt\Rm\ast\n g\leq c_2k_1f_1 \\
&g^{-1}\ast g^{-1}\ast \n g\ast\n g\ast\n^2 g\leq c_2f_1^2\cdot f_2\\
&g^{-1}\ast g^{-1}\ast g^{-1}\ast \n g\ast \n g\ast \n g\ast\n g\leq c_2f_1^4. \end{split}\eee

Then \eqref{boxdg} implies \be \label{boxdg1} \begin{split}
\square(f_1^2)\leq& -f_2^2+c_2\lf(2k_0f_1^2+k_1f_1+f_1^2f_2+f_1^4\ri)\\
\leq&-\frac{1}{2}f_2^2+c_3\lf(f_1^4+1\ri), \\
 \end{split} \ee
where we have used the assumption that $k_0+k_1\le 1$.
Next we define a smooth function $\varphi$ on $M\times[0,T]$ as follows:
\bee
\varphi= a(n)+g_{j_1i_1}h^{i_1j_2}g_{j_2i_2}h^{i_2j_3}\cdots g_{j_mi_m}h^{i_mj_1}. \eee
We will choose $a>0$ and $m$ later with $a$ depending only on $n$ and $m$ depending only on $n, \a$. One can choose a coordinate system $\{x^i\}$ such that at one point:
$ h_{ij}=\delta_{ij} $ and $g_{ij}=\lambda_i\delta_{ij}$. Then $\varphi=a+\sum\limits_{i=1}^{n}\lambda^m_i$. By direct computation, we have \be\label{boxphi}\begin{split}
\square\varphi=&m\lambda^{m-1}_k\ast(\wt\Rm\ast g^{-1}\ast g+g^{-1}\ast g^{-1}\ast\n g \ast\n g)\\
&-m(\lambda^{m-2}_i+\lambda^{m-3}\lambda_j+\cdots+\lambda^{m-2}_j)g^{\alpha\beta}\n_\alpha g_{ij}\n_{\beta}g_{ij}\\
\leq&c_4me^{2b(m-1)}\lf(f_1^2+ 1\ri)-\frac{m(m-1)}{2}e^{-2b(m-2)}f_1^2.
\end{split} \ee
Here we use the fact that $e^{-2b}h\leq g(t)\leq e^{2b}h$ and $k_0\le 1$.

Now we define $\psi=\varphi\cdot f_1^2$. By \eqref{boxdg1} and \eqref{boxphi}, we have \be\label{boxpsi}\begin{split}
\square\psi=&\varphi\cdot\square(f_1^2)+\square\varphi\cdot f_1^2-2g^{ij}\n_i \varphi\n_j f_1^2\\
\leq&\varphi\lf(-\frac{1}{2}f_2^2+c_3\lf(f_1^4+1\ri)\ri) +f_1^2 \lf(c_4me^{2b(m-1)}\lf(f_1^2+1\ri)-\frac{m(m-1)}{2}e^{-2b(m-2)}f_1^2\ri)
\\
&+\frac{\varphi}{2}f_2^2+\frac{c_5m^2e^{2b(m-1)}}{a}f_1^4\\
\leq&\varphi\lf(c_3\lf(f_1^4+1\ri)\ri) +f_1^2 \lf(c_4me^{2b(m-1)}\lf(f_1^2+1\ri)-\frac{m(m-1)}{2}e^{-2b(m-2)}f_1^2\ri)
\\&+\frac{c_5m^2e^{2b(m-1)}}{a}f_1^4,\\
\end{split} \ee
where we have used the fact \eqref{e-fair} and $e^{2b}\le 2$ which also imply
\bee
-2g^{ij}\n_i \varphi\n_j f_1^2\le cm(\sum_{i=1}^n\lambda_i^{m-1})f_1^2 f_2
\eee
for some constant $c$ depending only on $n$. Now let $b$, $m$ be such that

({\bf c2}) $b=\frac1{2m}$ and $m\ge 2$ with $e^{1/m}\le 2$.

Note that if $m\geq 2$, we have $m-1\geq \frac{m}{2}$, and $e^{bm}=e^{1/2}$. Hence the above inequality becomes:

  \bee\begin{split}
\square\psi\leq& c_7(a+1)\lf(f_1^4+1\ri)+c_7mf_1^2 \lf(1+ f_1^2\ri)-c_8m^2f_1^4+\frac{c_7m^2}{a}f_1^4. \end{split}
\eee
Let $a=\frac{2c_7}{c_8}$, then $a=a(n)$ which depends only on $n$. We have
 \be\label{e-d1-1}
 \begin{split}
\square\psi\leq& c_9\lf(mf_1^4+mf_1^2+1\ri)-\frac12c_8m^2f_1^4.
\end{split}
\ee
 Since $h$ has bounded curvature, there is a smooth function $\rho(x)$ such that
 $$
 d(p,x)+1\le \rho(x)\le D_1\lf(d(p,x)+1\ri), |\nabla\rho|+|\nabla^2\rho|\le D_1
 $$
 where $d(p,x)$ is the distance function from $p$ with respect to $h$ and $D_1$ is a constant depending on $h$, see \cite{Shi1997,Tam2010}. Let $\bar\eta(s)$ be a smooth function   on $\mathbb{R}$ such that $0\leq\bar\eta\leq 1$, $\bar\eta=1$ for $s\leq 1$, $\bar\eta=0$ for $s\geq 2$, $|\bar\eta'|^2\leq D_2\bar\eta$ and $|\bar\eta''|\leq D_2$.
 For any $r\ge 1$, let
 $$
 F(x,t)=t\eta_r(x)\psi(x)=t\eta_r(x)\varphi(x) f_1^2(x),
 $$
  where $\eta_r(x)=\bar\eta(\frac{\rho(x)}r)$.
By \eqref{e-d1-1}, we have
\bee\begin{split}
&\square F\le t\eta_r \lf[c_9\lf(mf_1^4+mf_1^2+1\ri)-\frac12c_8m^2f_1^4\ri]+\eta_r \psi-t\psi \square \eta_r-2tg^{ij}\n_i\eta_r\n_j \psi.
\end{split}
\eee
Suppose let $F(x_0,t_0)=\max_{(x,t)\in M\times[0,T]}F(x,t)$. Suppose $t_0=0$, then $F(x_0,t_0)=0$. Suppose $t_0>0$, then at $(x_0,t_0)$, $\psi\n_j\eta_r+\eta_r\n_j\psi=0$, and
multiplying the about inequality by $t_0\eta_r$, we have
\bee
\begin{split}
0\le &(t_0\eta_r)^2 \lf[c_9\lf(mf_1^4+mf_1^2+1\ri)-\frac12c_8m^2f_1^4\ri]+t_0\eta_r^2 \psi+D_3r^{-1}t_0^2\eta_r\psi\\
\le& c_{10}(mF^2+mF+1)-c_{11}m^2 F^2+(1+D_3r^{-1})F
\end{split}
\eee
where we have used the fact that $t_0\le T\le 1$, $\eta_r\le 1$, and $c^{-1}\le \varphi\le c$ for some constant $c$ depending only on $n$. Let $m$ be such that

({\bf c3}) $m\ge \frac{2c_{10}}{c_{11}}$.

Then at $(x_0,t_0)$

\bee
0\le -\frac12 c_{11}m^2F^2+\lf(c_{10}m+1+D_3 r^{-1}\ri)F+c_{10}
\eee
 and
 \bee
 F(x_0,t_0)\le \frac{2\lf(c_{10}m+1+D_3 r^{-2}\ri)+\lf(2c_{10} c_{11}m^2\ri)^\frac12}{c_{11}m^2}.
 \eee
Let $r\to\infty$, we conclude that:
\bee
\sup_{M\times[0,T]}t\varphi f_1^2\le \frac{c_{12}}m,
\eee
and so
\be\label{e-d1-2}
\sup_{M\times[0,T]}t|\nabla g|^2\le \frac{c_{13}}m\le \a
\ee
 provided

 ({\bf c4}) $m\ge \frac{c_{13}}\a$.

 Hence if we choose $m$ large enough so that $m$ satisfies ({\bf c3}), ({\bf c4}) such that $e^{1/m}\le 2$. Note that $m$ depends only on $n$ and $\a$. Then choose $b=\frac1{2m}$ and satisfies ({\bf c1}). $b$ also satisfies   ({\bf c2}). Then $b$ depends only on $n, \a$. For this choice of $m, b$ we conclude the lemma is true by \eqref{e-d1-2}.

 \end{proof}

 \begin{lma}\label{l-d2-estimate} Let $(M^n,h)$ be a complete noncompact Riemannian manifold such that $|\Rm(h)|\le k_0$,  $|\nabla\Rm(h)|\le k_1$, $|\nabla^2\Rm(h)|\le k_2$ with $k_0+k_1+k_2\leq 1$. For any $\a>0$ there is a constant $b(n,\a)>0$ depending only on $n$ and $\a$ such that $e^{2b}\le 1+\e(n)$ where $\e(n)$ is the constant in Theorem \ref{t-Simon}, and if $g(t)$ is the solution of the $h$-flow on $M\times [0,T]$, $T\leq 1$ obtained in Theorem \ref{t-Simon} with $g(0)=g_0$ satisfies $e^{-b}h\le g_0\le e^bh$, then there is a $T\ge T_2(n,\a)>0$ depending only on $n,\a$ such that
 \bee
 |\nabla^2 g(t)|^2\le \frac {\a^2} {t^2}
 \eee
 for all $t\in (0,T_2]$.
 \end{lma}

\begin{proof} As in the proof of the previous lemma, let $f_i=|\nabla^ig|$.  Let $b>0$ be such that $e^b\le 1+\e$ where $\e=\e(n)$ is the constant in Theorem \ref{t-Simon}. Suppose $h$ is $e^b$ close to $g_0$ and let $g(t)$ be the solution of the $h$-flow with $g(0)=g_0$ obtained in Theorem \ref{t-Simon}. Let $\beta>0$ to be chosen later depending only on $n, \a$. By Lemma \ref{l-d1-estimate}, we assume that:

({\bf c1}) $e^{2b}\le 2$ is such that there is $T\ge T_1(n,\beta)>0$ and $f_1^2\le \frac{\beta}t$ on $M\times[0,T_1]$.

Note that $b$ depends only on $n$ and $\beta$.  We may also assume that
$$
e^{-2b}h\le g\le e^{2b}g.
$$
as in the proof of the previous lemma. As before, in the following, lower case $c_i$ will denote a positive constant depending only on $n$.

 By \eqref{boxddg}, we have
 \be\label{boxddg1}
 \begin{split}
\square f_2^2 \leq& -f_3^2+ c_1( f_2+ f_1 f_2+ f_1^2 f_2+f_2^2+ f_1f_2f_3 +f_2^3+f_1^2f_2^2+f_1^4f_2)\\
\le &-\frac12 f_3^2+c_2\lf( f_2+ f_1 f_2+ f_1^2 f_2+f_2^2+f_2^3+f_1^2f_2^2+f_1^4f_2\ri) \\
\le &-\frac12 f_3^2+c_3\lf(f_2^2+f_2^3+f_1^2f_2^2+f_1^2+f_1^4+f_1^6+1\ri)\\
\le &-\frac12 f_3^2+c_4\lf(f_2^3+\frac\b t f_2^2+\lf(\frac\b t\ri)^3+1\ri)
\end{split}
\ee
provided that

({\bf c2}) $t\le \beta$.

Here we have used that fact that $f_0\le c$, $f_1\le \frac{\beta}t$ and $k_0+k_1+k_2\le 1.$
In the following, we always assume that ({\bf c2}) is true.

Let $\psi(x,t)=\lf(at^{-1}+f_1^2\ri)f_2^2$, where $a>0$ is a constant depending only on $n$ and $\beta$ and will be chosen later.

Combine \eqref{boxdg1} and \eqref{boxddg1}, we have
\be\label{e-ddg-1}
\begin{split}
\square\psi=&   (at^{-1}+f_1^2)\square f_2^2+f_2^2\square f_1^2  -2g^{ij}\n_i(f_1^2)\cdot\n_j(f_2^2)-at^{-2}f_2^2\\
\leq
& (at^{-1}+f_1^2)\lf(-\frac12 f_3^2+c_4\lf(f_2^3+\frac\b t f_2^2+\lf(\frac\b t\ri)^3+1\ri)\ri)\\
&+f_2^2\lf(-\frac{1}{2}f_2^2+c_5\lf(f_1^4+1\ri)\ri) +c_5f_1f_3f_2^2 \\
\leq& \lf(c_6f_1^2-\frac12at^{-1}\ri)f_3^2+c_4(at^{-1}+f_1^2)\lf(f_2^3+\frac\b t f_2^2+\lf(\frac\b t\ri)^3+1\ri)\\
&+f_2^2\lf(-\frac{1}{4}f_2^2+c_5\lf(f_1^4+1\ri)\ri)   \\
\le &-\frac18 f_2^4+c_6\lf(\frac\b t\ri)^4
 \end{split}
\ee
where we have chosen $a$ so that

({\bf c3}) $a=2c_6\beta$ which depends only on $n$ and $\beta$.

Here we have used the fact that $t\le \b$.

For $r>1$, let  $\rho$, $\bar\eta$, $\eta_r$ as in the proof of the previous lemma, and let $F=t^p\eta_r\psi=t^p\eta_r(at^{-1}+f_1^2)f_2^2$, $p\ge2$. Let
$$
F(x_0,t_0)=\max_{(x,t)\in M\times[0,T_1]}F(x,t).
$$
If $t_0=0$, then $F(x_0,t_0)=0$. If $t_0>0$, then at $(x_0,t_0)$, we have by \eqref{e-ddg-1}, as in the proof of the previous lemma:

\bee
\begin{split}
0\le& t_0^p\eta_r\lf(-\frac18 f_2^4+c_6\lf(\frac\b {t_0}\ri)^4\ri)+pt_0^{p-1}\eta_r\psi+t_0^p\psi\square \eta_r
-2t_0^p\n_i\eta_r\n_j\psi\\
\le&t_0^p\eta_r\lf(-\frac18 f_2^4+c_6\lf(\frac\b {t_0}\ri)^4\ri)+pt_0^{p-1}\eta_r\psi+D_1r^{-1}t_0^p\psi
\end{split}
\eee
where $D_1$ depends on $h$. Multiply both sides   by $t_0^p\eta_r(at_0^{-1}+f_1^2)^2$ using the fact that $t_0\le 1, \eta_r\le 1$ and that $t_0\le \beta$, we have
\bee
\begin{split}
 \frac18 F^2\le&c_6t_0^{2p}\eta_r (at_0^{-1}+f_1^2)^2\lf(\frac\b {t_0}\ri)^4+ pt_0^{2p-1}\eta_r^2(at_0^{-1}+f_1^2)^2\psi+D_1r^{-1}t_0^{2p}\eta_r(at_0^{-1}+f_1^2)^2\psi\\
 \le &c_7t_0^{2p}\lf(\frac\b {t_0}\ri)^6+c_8F\lf(\frac\b {t_0}\ri)^2\lf(pt_0^{p-1}+D_1r^{-1}t_0^p\ri).
 \end{split}
\eee
Let $p=3$, we have
\bee
\frac18 F^2\le c_9\lf(\b^6+\b^2 F\lf(1+D_1r^{-1}\ri)\ri)
\eee
Hence
\bee
F(x_0,t_0)\le c_{10}\lf(\b^3+\b^2(1+D_1r^{-1})\ri).
\eee
Let $r\to\infty$, we conclude that
$$
\sup_{(x,t)\in M\times[0,T_2]}t^3(at^{-1}+f_1^2)f_2^2\le c_{10}\lf(\b^3+\b^2\ri).
$$
where $T_2=\min\{T_1,\b\}$. By the definition of $a$, we conclude that
$$
t^2|\nabla^2g|^2\le c_{11}\b
$$
provided $\b\le 1$. Now choose $\b\le1$ such that $c_{11}\b<\a^2$. $\b$ depends only on $n, \a$. Choose choose $b$ satisfying ({\bf c1}) and $a$ satisfying ({\bf c3}). If $t\le \b, T_1$, then we have
$$
t^2|\nabla^2g|^2\le \a^2
$$
in $M\times[0,T_2]$, where $T_2=\min\{T_1,\b\}$. This completes the proof of the lemma.

\end{proof}

 \begin{lma}\label{l-Ricciflow}
 For any $\a>0$, there exists $\e(n,\a)>0$ depending only on $n$ and $\a$ such that if $(M^n,g_0)$ is a complete noncompact  Riemannian manifold with real dimension $n$ and if $g_0$  is $(1+\e)$ close to a Riemannian metric $h$ with curvature bounded by $k_0$, then there is a smooth complete Ricci flow $g(t)$   defined on $M\times[0,T]$ with initial value $g(0)=g_0$, where $T>0$ depends only on $n, k_0$. Moreover, there is $T_1(n,k_0,\a)>0$ depending only on $n, \a$ such that the curvature of $g(t)$ satisfies:
 $$
 |\Rm(g(t))|_{g(t)}\le \frac \a t
 $$
on $M\times[0,T_1]$.
 \end{lma}

\begin{proof} First we remark that by \cite{Shi1989}, there is a solution to the Ricci flow with initial data $h$  with bounded curvature in space and time. Moreover, for $t>0$ all order of derivatives of the curvature tensor for a fixed $t>0$ are uniformly bounded, the solution exists in a time interval depending only on $n, k_0$, and the bounds of the derivatives of the curvature tensor for a fixed $t>0$ depend only on $n, k_0,$ and $t$. Hence without lost of generality, we may assume that $|\tn^{(i)}\wt\Rm|_{h}\le k_i<\infty$ for all $i\ge0$. Here and in the following   $\tn$ is the covariant derivative with respect to $h$ and $ \wt\Rm$ is the curvature tensor of $h$ and $|\,\cdot\,|_{h}$ is the norm relative to $h$.

Note that if $h$ is $1+\e$ close to $g_0$, then $\lambda h$ is also $1+\e$ close to $\lambda g_0$ for any $\lambda>0$. Moreover, if $g(t)$ is a solution to the Ricci flow with initial data $g_0$, then $\lambda g(\lambda^{-1}t)$ is a solution to the Ricci flow with initial data $\lambda g_0$, and if $s=\lambda t$, then
$$
|\Rm(g(t))|_{g(t)}=\lambda|\Rm(\lambda g(\lambda^{-1}s))|_{\lambda g(\lambda^{-1}s)}.
$$
Hence we may assume that $k_0+k_1+k_2\le 1$.

Let us first assume that $\e(n,\a)<\e(n)$ where $\e(n)$ is the constant in Theorem \ref{t-Simon}.

For any $R>0$, let $0\le \eta_R$ be a smooth function on $M$ such that
\bee
\eta_R=\left\{\begin{array}{cl}
       1, & x\in B_{g_0}(x_0,R)  \\
       0, & x\in M\setminus B_{g_0}(x_0,2R)
     \end{array} \right.
\eee
Let $g_{R,0}=\eta_Rg_0+(1-\eta_R)h$. Let $\b>0$ to be chosen later depending only on $n, \a$. Suppose $h$ is $e^b=(1+\e)$ close to $g_0$, where $b>0$, then one can see that $h$ is also $e^b$ close to $g_{R,0}$ for all $R>0$.
By Lemmas \ref{l-d1-estimate}, \ref{l-d2-estimate}, there is a constant $b>0$ depending only on $n, \b$ such that for all $R>0$ the solution $\bar g_R(t)$ to the $h$-flow as in Theorem \ref{t-Simon} exists on $M\times[0,T_2]$ for some $T_2>0$ depending only on $n, \b$ such that
\be\label{e-shorttime-1}
|\tn^i \bar g_R(t)|^2_{h}\le \frac{C_i}{t^i}
\ee
for all $i\ge 0$, where $C_i$ depends only on $n, i, k_0,\dots, k_i$. Moreover,
\be\label{e-shorttime-2}
|\bar g_R(t)|_{h}\le 2, |\tn^i \bar g_R(t)|^2_{h}|\le \frac{\b^i}{t^i} \ \ \text{for $i=1, 2$.}
\ee

Now we want to claim that there is a constant $c_1=c_1(n)$ depending only on $n$ such that
\be\label{e-Rm}
|\Rm(\bar g_R(t))|_{g_R(t)}\le c_1\lf(|\wt \Rm|_{h}+|\tn \bar g_R(t)|^2_{h} +|\tn^2 \bar g_R(t)|_{h}\ri).
\ee

To see \eqref{e-Rm}, we choose a normal coordinate at any fix point $x\in M$ with respect to $h$ and it also diagonalizes  $\bar g_R(t)$. In this coordinate, we have \bee \begin{split}
\bar R^l_{ijk}=&\frac{\partial}{\partial x^i}\bar\Gamma^l_{kj}-\frac{\partial}{\partial x^j}\bar\Gamma^l_{ki}+\bar\Gamma^h_{kj}\bar\Gamma^l_{hi}-\bar\Gamma^h_{ki}\bar\Gamma^l_{hj}\\
=&\frac{\partial}{\partial x^i}(\bar\Gamma^l_{kj}-\wt\Gamma^l_{kj})-\frac{\partial}{\partial x^j}(\bar\Gamma^l_{ki}-\wt\Gamma^l_{ki})\\
&+\frac{\partial}{\partial x^i}\wt\Gamma^l_{kj}-\frac{\partial}{\partial x^j}\wt\Gamma^l_{ki}+\bar\Gamma^h_{kj}\bar\Gamma^l_{hi}-\bar\Gamma^h_{ki}\bar\Gamma^l_{hj}\\
=&\wt R^l_{kij}+\wt\n_i(\bar\Gamma^l_{kj}-\wt\Gamma^l_{kj})-\wt\n_j(\bar\Gamma^l_{ki}-\wt\Gamma^l_{ki})+\bar\Gamma^h_{kj}\bar\Gamma^l_{hi}-\bar\Gamma^h_{ki}\bar\Gamma^l_{hj}.\end{split} \eee
Here we use $\bar \cdot$ to denote the Christoffel symbol and curvature tensor of the metric $\bar g_R(t)$.
Note that \bee \bar\Gamma^l_{kj}-\wt\Gamma^l_{kj}=\frac{1}{2}g^{ls}(\wt\n_kg_{js}+\wt\n_jg_{ks}-\wt\n_sg_{kj}), \eee
we have \bee \wt\n_i(\bar\Gamma^l_{kj}-\wt\Gamma^l_{kj})=g^{-1}\ast g^{-1}\wt\n g\ast\wt\n g+g^{-1}\ast\wt\n^2 g \eee and \bee
\bar\Gamma\ast\bar\Gamma=g^{-1}\ast g^{-1}\ast\wt\n g\ast\wt\n g. \eee

Therefore, we obtain \eqref{e-Rm}.

Then, we have
\be\label{e-shorttime-3}
|\Rm(\bar g_R(t))|_{g_R(t)}\le\frac{3c_1\beta}t
\ee
provided on $M\times[0,T_2]$ provided $T_2$ is small depending only on $n, \beta, k_0$. Here we have used the fact that $h$ is $1+2\e(n)$ close to $\bar g_R(t)$.

 Using the similar argument as in  the proof of \cite[Lemma 4.1]{Simon2002} or Lemma \ref{l-timederivative-1}, one can show that
 $$
 |\tn  \bar g_R(t)|\le C(n,h,g_0,R)
 $$
 for some constant $C(n,h,g_0,R)$ depending only on $n, h, g_0, R$. Hence  we can pull back $\bar g_R(t)$ by a smooth family   of diffeomorphisms from $M$ to itself $\varphi_R(t)$, $t\in[0,T]$. That is, let $g_R(t)=\varphi_R(t)^*\bar g_R(t)$ on $M\times[0,T_2]$ where  $\varphi_R(t)$, $t\in[0,T_2]$ is given by solving the following ODE at each point $x\in M$:\be\label{ode} \left\{\begin{array}{cc}
\frac{d}{dt} \varphi_R(x,t)=-W(\varphi_R(x,t),t) \\
\varphi_R(x,0)=x
\end{array} \right.
\ee
where $W$ is a time-dependent smooth vector field
given by \bee W^i(t)=\bar g_R^{jk}(t)({}^{\bar g_R(t)}\Gamma^i_{jk}-{}^{h}\Gamma^i_{jk}).
\eee
Then $g_R(t)$ is a solution to the Ricci flow with $g_R(0)=g_{R,0}$. By \eqref{e-shorttime-3}

 \bee\label{1}
|Rm(g_R(t))|_{g_R(t)}\leq \frac{3c_1\b }{t} \eee
 on $M\times(0,T_2]$. Since $g_{R,0}$ has uniformly bounded curvature, which may depends on $R$, by \cite{Chen,Simon2008} for any compact set $U$, there is a constant $C_1$ independent of $R$ such that

\bee
|Rm(g_R(t))|_{g_R(t)}\leq C_1\eee on $  U\times[0,T_2]$. By \cite{Shi1989} (see also \cite[Theorem 11]{LuTian}), we see that for each $m$, there is a constant $C(m)$ independent of $R$ such that
\bee
|{}^{g_R(t)}\n^mRm_{g_R(t)}|_{g_R(t)}\leq C(m) \eee on $  U\times [0,T_2]$. From this, we obtain that \bee
|{}^{g_R(t)}\n^mg_R(t)|_{g_R}\leq C(m) \eee on $U\times[0,T_2]$ for some constant $C(m)$ independent of $R$. Hence by diagonal process, passing to a subsequence, $g_R(t)$ converges in $C^\infty$ topology on compact sets of $M\times[0,T_2]$ to a solution $g(t)$ of the Ricci flow with $g(0)=g_0$. Moreover, by \eqref{e-shorttime-3},
$$
|\Rm(g(t))|_{g(t)}\le \frac{3c_1\b}t.
$$

Next, we claim $g(t)$ is complete for all $t\in[0,T_2]$.
Let $\{y_k\}$ be a divergence  sequence of points in  $M$.
For any fixed point $x_0$ and $t\in[0,T_2]$, we have
\bee\begin{split}
d_{g_R(t)}(x_0,y_k)=&d_{\bar g_R(t)}(\varphi_R(x_0,t),\varphi_R(y_k,t))\\
\geq&d_{\bar g_R(t)}(x_0,y_k)-d_{\bar g_R(t)}(\varphi_R(x_0,t),x_0)-d_{\bar g_R(t)}(y_k,\varphi_R(y_k,t)), \end{split} \eee
for some positive constants $C_3, C_4$ independent of $R, y_k$, where we have used \eqref{e-shorttime-1} which implies $W(x,t)$ in \eqref{ode} is uniformly bounded by a $Ct^{-\frac12}$ for some constant $C$ for all $x, t$ and $ R$, and we have also used the fact that $(1+2\e)^{-1}h\le g_R\le (1+2\e)h$ for all $R$ and $t\in [0,T]$. This implies \bee d_{g_R(t)}(x_0,y_n)\geq C_3d_h(x_0,y_n)-C_4\sqrt{T_2}, \eee

Let $R\to +\infty$, we see that
\bee
d_{g(t)}(x_0,y_k)\geq C_3d_h(x_0,y_n)-C_4\sqrt{T_2}.
\eee
Since $h$ is complete,   we obtain $
d_{g(t)}(x_0,y_k)\to  \infty$ as $k\to\infty$.  This implies $g(t)$ is complete.

 Now choose $\b$ such that $3c_1\b=\a$, we conclude that the lemma is true.
\end{proof}

 Now we want to prove the main result of this section:

 \begin{thm}\label{t-existence-fair}
 There exists $\e(2n)>0$ depending only on $n$ such that if $(M^n,g_0)$ is a complete noncompact  \K manifold with complex dimension $n$ and if there is a smooth Riemannian metric $h$ with   curvature bounded by $k_0$ on $M$ such that $g_0$  is $(1+\e(n))$ close   $h$, then there is a complete \KR flow $g(t)$ defined on $M\times[0,T]$ with initial value $g(0)=g_0$, where $T>0$ depends only on $n, k_0$. Moreover, the curvature of $g(t)$ satisfies:
 $$
 |\Rm(g(t))|_{g(t)}\le \frac \a t
 $$
 where $\a=\a(n)$ is the constant in Theorem \ref{nonnegative-bi}. If in addition, $g_0$ has nonnegative holomorphic bisectional curvature, then $g(t)$ has nonnegative holomorphic bisectional curvature for all $t\in [0,T]$.
 \end{thm}
\begin{proof}
The results follow from Lemma \ref{l-Ricciflow} and Theorems \ref{t-Kahler}, \ref{nonnegative-bi}.
\end{proof}

\begin{cor}\label{c-uniform-hflow} Let $\e(2n)$ be as in Theorem \ref{t-existence-fair}. Suppose $(M^n,g_0)$ is a complete noncompact \K manifold with complex dimension $n$ with nonnegative holomorphic bisectional curvature with maximum volume growth. Suppose there is a Riemannian metric $h$ on $M$ with bounded curvature which is $1+\e(2n)$ close to $g_0$. Then $M$ is biholomorphic to $\C^n$.
\end{cor}
\begin{proof} Let $g(t)$ be the solution of \KR flow obtained in Theorem \ref{t-existence-fair}. Then for $t>0$, $g(t)$ is \K with bounded nonnegative holomorphic bisectional curvature. We claim that $g(t)$ has maximum volume growth. Let $x_0\in M$ be fixed. By the proof of Lemma \ref{l-Ricciflow}, using the same notations as in the proof we conclude that
\be\label{e-vol}
V_{\bar g_R(t)}(x_0,r)\ge C_1r^{2n}
\ee
for some $C_1>0$ for all $r$ because $g_0$ has maximum volume growth and $\bar g_R(t)$ is uniformly equivalent to $h$ which in turn is uniformly equivalent to $g_0$. Here
$V_{\bar g_R(t)}(x_0,r)$ is the volume of the geodesic ball $B_{\bar g_R(t)}(x_0,r)$ with respect to $\bar g_R(t)$. As in the proof of Lemma \ref{l-Ricciflow},
$$
V_{\bar g_R(t)}(x_0,r)=V_{g_{R}(t)}(\varphi^{-1}_{t}(x_0),r)\leq V_{g_{R}(t)}(x_0,r+C_2)
$$
for some constant $C_2>0$ independent of $R$ and $x_0$. From this and \eqref{e-vol}, we conclude that $g(t)$ has maximum volume growth. Hence $M$ is biholomorphic to $\C^n$ by the result of \cite{ChauTam2008}.

\end{proof}


\begin{thebibliography}{1000}
\bibitem{Bando1984} Bando, S.,{\sl On the classification of three-dimensional compact Kaehler manifolds of nonnegative bisectional curvature}, J. Differential Geom.  \textbf{19}  (1984),  no. 2, 283--297.
\bibitem{Cabezas-RivasWilking2011}   Cabezas-Rivas, E.;   Wilking, B., {\sl How to produce a Ricci Flow via Cheeger-Gromoll exhaustion} arXiv:1107.0606 (2011)
\bibitem{ChauLiTam1}    Chau, A.;  Li, K.-F.;  Tam, L.-F., {\sl Deforming complete Hermitian metrics with unbounded curvature}, to appear in Asian J. Math, arXiv:1402.6722
 \bibitem{ChauLiTam2}    Chau, A.;  Li, K.-F.;  Tam, L.-F., {\sl Longtime existence of the Kähler-Ricci flow on $\Bbb C ^n$}, arXiv:1409.1906.

\bibitem{ChauTam2008} Chau, A. ; Tam, L.-F.,
 {\sl On the complex structure of K\"{a}hler manifolds with non-negative
 curvature}, J. Differential Geom. \textbf{73} (2006),   491--530, MR 2228320

\bibitem{Chen} Chen, B.L., {\sl Strong Uniqueness of the Ricci Flow}, J. Differential
Geom. 82 (2009), no. 2, 363--382.

\bibitem{CTZ} Chen, B.L.; Tang, S.H.; and Zhu, X.P.,
 {\sl A Uniformization Theorem Of Complete Noncompact
       K\"{a}hler Surfaces With Positive Bisectional Curvature}, J. Differential Geom.  67  (2004),  no. 3, 519–570.


\bibitem{Chow07} Chow, B.; Chu, S.-C.; Glickenstein, D.; Guenther, C.; Isenberg, J.; Ivey, T.; Knopf, D.; Lu, P.; Luo, F.; Ni, L.,
 {\sl The Ricci flow : techniques and applications, Part III. Geometric-analytic aspects}
Mathematical Surveys and Monographs, {\bf 135},  American Mathematical Society, Providence, RI, (2007)



 \bibitem{ChowKnopf} Chow, B.; Knopf, D., {\sl The Ricci flow: an introduction. Mathematical Surveys and Monographs}  {\bf 110}, American Mathematical Society,  2004.
 \bibitem{GT} G. Giesen; P.M. Topping, {\sl Existence of Ricci flows of incomplete surfaces}, Comm. Partial Differential Equations \textbf{36} (2011), no. 10, 1860--1880.
 \bibitem{Hamilton1982} Hamilton, R. S., {\sl Three-manifolds with positive Ricci curvature}, J. Differential Geom. \textbf{ 17}  (1982), no. 2, 255--306.
 \bibitem{Liu2015}    Liu, G.,{\sl Gromov-Hausdorff limit of Kähler manifolds and the finite generation conjecture}, arXiv:1501.00681.
\bibitem{LuTian} Lu, P.; Tian, G., {\sl Uniqueness of standard solutions in the work of Perelman}  Available at www.math.lsa.umich.edu/~lott/ricciflow/StanUniqWork2.pdf


\bibitem{Mok1984}   Mok, N., {\sl An embedding theorem of complete \K manifolds of positive bisectional curvature onto affine algebraic varieties}, Bull. Soc. Math. France  \textbf{112}  (1984),  no. 2, 197--250.

\bibitem{Mok}  Mok, N., {\sl The Uniformization Theorem for Compact K\"ahler manifolds of Nonnegative Holomorphic Bisectional Curvature} J. Differential Geometry, \textbf{27} (1988) 179--214



\bibitem{Perelman-1} Perelman, G., {\sl The entropy formula for the
Ricci flow and its geometric applications}, arXiv:math.DG/0211159


\bibitem{Petrunin} Petrunin, A. M., {\sl
An upper bound for the curvature integral},  Algebra i Analiz  \textbf{20}  (2008),  no. 2, 134--148;  translation in
St. Petersburg Math. J.  \textbf{20}  (2009),  no. 2, 255--265

\bibitem{Simon2002}  Simon, M., {\sl Deformation of $C^0$ Riemannian metrics in the direction of their Ricci curvature},
Comm. Anal. Geom. \textbf{10}(2002), no. 5, 1033--1074
\bibitem{Simon2008}   Simon, M., {\sl Local results for flows whose speed or height is bounded by $c/t$},
Int. Math. Res. Not. IMRN  (2008), Art. ID rnn \textbf{097}, 14 pp


\bibitem{Shi1989}  Shi, W.-X., {\sl Deforming the metric on complete Riemannian manifolds}, J. Differential
Geom. 30 (1989), no. 1, 223--301

\bibitem{Shi1997} Shi, W.-X., {\sl Ricci Flow and the uniformization on complete non compact \K
manifolds}, J. of Differential Geometry \textbf{45} (1997), 94--220




\bibitem{Tam2010} Tam, L.-F.,  {\sl Exhaustion functions on complete manifolds},  211--215 in   Recent advances in geometric analysis,  Adv. Lect. Math. (ALM), \textbf{11}, Int. Press, Somerville, MA, 2010

\bibitem{Xu2013} Xu, G., {\sl
Short-time existence of the Ricci flow on noncompact Riemannian manifolds},
Trans. Amer. Math. Soc. \textbf{ 365}  (2013),  no. 11, 5605--5654.
\bibitem{YangZheng2013} Yang, B.; Zheng, F.-Y., { $U(n)$-invariant \KR flow with non-negative curvature}, Comm. Anal. Geom. \textbf{ 21 } (2013),  no. 2, 251--294.
\bibitem{Y} Yau, S.-T., {\sl A review of complex differential geometry},    Proc. Sympos. Pure Math., \textbf{52} Part 2 (1991), 619--625.

\bibitem{Zheng1995} Zheng, F.-Y.,{\sl First Pontrjagin form, rigidity and strong rigidity of nonpositively curved \K surface of general type}, Math. Z.  \textbf{220}  (1995),  no. 2, 159--169.


























\end{thebibliography}
\end{document}